\documentclass[12pt]{article}

\usepackage{amsmath} 
\usepackage{amssymb} 
\usepackage{theorem}

\theoremstyle{change} 
\newtheorem{theorem}{Theorem}[section] 
\newtheorem{lemma}[theorem]{Lemma} 
\newtheorem{proposition}[theorem]{Proposition}
\newtheorem{corollary}[theorem]{Corollary}
\theorembodyfont{\rmfamily} 

\newtheorem{remark}[theorem]{Remark}
\newtheorem{example}[theorem]{Example}

\newtheorem{nothing}[theorem]{} 

\newtheorem{remark and notation}[theorem]{Remark and Notation}

\newenvironment{proof}{\noindent{\bf Proof}\ }{\qed\bigskip}

\renewcommand{\le}{\leqslant}

\newcommand{\br}{\mathrm{br}}
\newcommand{\btilde}{\tilde{b}}

\newcommand{\calO}{\mathcal{O}}

\newcommand{\catfont}{\mathsf}
\newcommand{\cbar}{\bar{c}}

\newcommand{\cdotH}{\mathop{\cdot}\limits_{H}}

\newcommand{\Con}{\mathrm{Con}}
\newcommand{\ctilde}{\tilde{c}}
\newcommand{\Def}{\mathrm{Def}}
\newcommand{\DefRes}{\mathrm{DefRes}}
\newcommand{\ebar}{\bar{e}}
\newcommand{\End}{\mathrm{End}}
\newcommand{\fbar}{\bar{f}}
\newcommand{\FF}{\mathbb{F}}
\newcommand{\gammatilde}{\tilde{\gamma}}
\newcommand{\Gal}{\mathrm{Gal}}

\newcommand{\hE}{\lexp{h}{E}}
\newcommand{\hJ}{\lexp{h}{J}}
\newcommand{\htilde}{\tilde{h}}
\newcommand{\Hom}{\mathrm{Hom}}

\newcommand{\Ind}{\mathrm{Ind}}
\newcommand{\Inf}{\mathrm{Inf}}
\newcommand{\Irr}{\mathrm{Irr}}
\newcommand{\Isom}{\mathrm{Iso}}
\newcommand{\kk}{\mathbb{\Bbbk}}
\newcommand{\KK}{\mathbb{K}}

\newcommand{\lexp}[2]{\setbox0=\hbox{$#2$} \setbox1=\vbox to
                 \ht0{}\,\box1^{#1}\!#2}
\newcommand{\lmod}[1]{\llap{\phantom{|}}_{#1}\catfont{mod}}
\newcommand{\lset}[1]{\llap{\phantom{|}}_{#1}\catfont{set}}
\newcommand{\Ltilde}{\tilde{L}}

\newcommand{\Mtilde}{\tilde{M}}

\newcommand{\myiso}{\buildrel\sim\over\to}

\newcommand{\nubar}{\bar{\nu}}
\newcommand{\omegabar}{\bar{\omega}}

\newcommand{\qed}{\nobreak\hfill
                   \vbox{\hrule\hbox{\vrule\hbox to 5pt
                   {\vbox to 8pt{\vfil}\hfil}\vrule}\hrule}}
\newcommand{\QQ}{\mathbb{Q}}

\newcommand{\Res}{\mathrm{Res}}
\newcommand{\rk}{\mathrm{rk}}

\newcommand{\tens}[3]{\mathop{\otimes}\limits^{{#1},{#2}}_{#3}}
\newcommand{\tenstop}[2]{\mathop{\otimes}\limits^{{#1},{#2}}}

\newcommand{\ZZ}{\mathbb{Z}}

\title{The Brou\'e invariant of a $p$-permutation equivalence\footnote{{\bf MR Subject Classification:}  
20C15, 20C20, 19A22. {\bf Keywords:}  Blocks of finite groups; perfect isometry; splendid Rickard equivalence; $p$-permutation equivalence; bisets.}}

\author{\small Robert Boltje\\
  \small Department of Mathematics\\
  \small University of California\\
  \small Santa Cruz, CA 95064\\
  \small U.S.A.\\
  \small boltje@ucsc.edu}
\date{\today}

\begin{document}
\sloppy


\maketitle


\begin{abstract}
A perfect isometry $I$ (introduced by Brou\'e) between two blocks $B$ and $C$ is a frequent phenomenon in the block theory of finite groups. It maps an irreducible character $\psi$ of $C$ to $\pm$ an irreducible character of $B$. Brou\'e proved that the ratio of the codegrees of $\psi$ and $I(\psi)$ is a rational number with $p$-value zero and that its class in $\FF_p$ is independent of $\psi$. We call this element the Brou\'e invariant of $I$. The goal of this paper is to show that if $I$ comes from a $p$-permutation equivalence or a splendid Rickard equivalence between $B$ and $C$ then, up to a sign, the Brou\'e invariant of $I$ is determined by local data of $B$ and $C$ and therefore, up to a sign, is independent of the $p$-permutation equivalence or splendid Rickard equivalence. Apart from results on $p$-permutation equivalences, our proof requires new results on extended tensor products and bisets that are also proved in this paper. As application of the theorem on the Brou\'e invariant we show that various refinements  of the Alperin-McKay Conjecture, introduced by Isaacs-Navarro, Navarro, and Turull are consequences of $p$-permutation equivalences or Rickard equivalences over a sufficiently large complete discrete valuation ring or over $\ZZ_p$, depending on the refinement.
\end{abstract}


\section{Introduction}\label{sec intro}

Throughout this introduction, we fix finite groups $G$ and $H$ and a complete discrete valuation ring $\calO$ containing a root of unity $\zeta$ of order $\exp(G\times H)$, the exponent of $G\times H$. We assume that the field of fractions $\KK$ of $\calO$ has characteristic $0$ and that its residue field $F=\calO/J(\calO)$ has prime characteristic $p$. We denote by $a\mapsto \bar{a}$ the canonical epimorphisms $\calO\to F$ and $\calO X\to FX$ for any finite group $X$.

Further we fix primitive central idempotents $b$ of $\calO G$ and $c$ of $\calO H$. We denote by $B:=\calO Gb$ and $C:=\calO Hc$ the corresponding block algebras.  By $R(\KK Gb,\KK Hc)$ we denote the Grothendieck group of $(\KK Gb,\KK Hc)$-bimodules and always view it via the usual category isomorphism $\lmod{\KK G}_{\KK H}\cong \lmod{\KK [G\times H]}$ as subgroup of the Grothendieck group $R(\KK[G\times H])$ of finitely generated $\KK[G\times H]$-modules.

\smallskip
In \cite{Broue1990}, Brou\'e introduced the notion of a {\em perfect isometry} between $B$ and $C$ as follows. An element $\mu\in R(\KK[G\times H])$ is called {\em perfect} if it satisfies the following two conditions:
\begin{itemize}
\item[(i)] For any $(g,h)\in G\times H$ one has $\mu(g,h)/|C_G(g)|\in\calO$ and $\mu(g,h)/|C_H(h)|\in\calO$.
\item[(ii)] If $(g,h)\in G\times H$ satisfies $\mu(g,h)\neq 0$ then $g$ is a $p'$-element if and only if $h$ is a $p'$-element.
\end{itemize}
Note that one has a group isomorphism $R(\KK Gb,\KK Hc)\myiso \Hom(R(\KK Hc), R(\KK Gb))$, $\mu\mapsto I_\mu$, induced by the functor $M\otimes_{\KK Hc}-\colon\lmod{\KK Hc}\to\lmod{\KK Gb}$ for any finitely generated $(\KK Gb,\KK Hc)$-bimodule $M$. 
A {\em perfect isometry} between $B$ and $C$ is an isomorphism $I=I_\mu\colon R(\KK Hc)\to R(\KK Gb)$ induced by a perfect element $\mu\in R(\KK Gb,\KK Hc)$, which respects the Schur inner products on $R(\KK G)$ and $R(\KK H)$. 
Thus, $I$ is a \lq bijection with signs\rq\ between $\Irr(\KK Hc)$ and $\Irr(\KK Gb)$: $I$ determines a bijection $\alpha\colon\Irr(\KK Hc)\myiso\Irr(\KK Gb)$ and signs $\varepsilon_\psi\in\{\pm1\}$, $\psi\in\Irr(\KK Hc)$, such that $I(\psi)=\varepsilon_\psi \alpha(\psi)$, for all $\psi\in\Irr(\KK Hc)$. If $I=I_\mu$ is a perfect isometry we also call $\mu$ a {\em perfect isometry}.

In \cite[Lemme~1.6]{Broue1990}, it is shown that if $\mu$ is a perfect isometry between $B$ and $C$ then the rational numbers
\begin{equation}\label{eqn Broue quotient}
 \frac{|G|/I_\mu(\psi)(1)}{|H|/\psi(1)}\,, \quad \psi\in\Irr(\KK Hc)\,,
\end{equation}
are units in the localization $\ZZ_{(p)}$ and their residue classes in $\ZZ_{(p)}/p\ZZ_{(p)}=\FF_p$ are equal, independent of $\psi$. We will denote this element in $\FF_p^\times$, which is uniquely determined by $\mu$, by $\beta(\mu)$ and will call it the {\em Brou\'e invariant} of $\mu$.

\medskip
By \cite[Proposition~1.2]{Broue1990}, an element $\mu\in R(\KK Gb,\KK Hc)$ is perfect if it belongs to the $\ZZ$-span of characters of indecomposable $p$-permutation $(B,C)$-bimodules $M$ (i.e., direct summands of permutation $\calO[G\times H]$-modules when viewed as $\calO[G\times H]$-module) which have a {\em twisted diagonal} vertices, i.e., vertices of the form $\Delta(P,\phi,Q):=\{(\phi(y),y)\mid y\in Q\}$, where $P\le G$ and $Q\le H$ are $p$-subgroups and $\phi\colon Q\myiso P$ is an isomorphism. We denote the free abelian group on the set of isomorphism classes $[M]$ of such indecomposable modules $M$ by $T^\Delta(B,C)$. In \cite{BP2020}, a {\em $p$-permutation equivalence} between $B$ and $C$ was defined to be an element $\gamma\in T^\Delta(B,C)$ with the property that $\gamma\cdot _H\gamma^\circ=[B]$, where $-^\circ\colon T^\Delta(B,C)\to T^\Delta(C,B)$ is induced by taking $\calO$-duals and $-\cdot_H-\colon T^\Delta(B,C)\times T^\Delta(C,B)\to T^\Delta(B,B)$ is induced by $-\otimes_{\calO H}-$. It follows that also $\gamma^\circ\cdot_G\gamma=[C]\in T^\Delta(C,C)$, see \cite[Theorem~12.3]{BP2020}. Using the canonical map 
\begin{equation*}
   \kappa\colon T^\Delta(B,C)\to R(\KK Gb,\KK Hc)
\end{equation*}
induced by $\KK\otimes_\calO -$, every $p$-permutation equivalence yields a perfect isometry $I_\mu$ with $\mu:=\kappa(\gamma)$, see \cite[Proposition~9.9]{BP2020}. We set $\beta(\gamma):=\beta(\kappa(\gamma))$ and call $\beta(\gamma)$ again the {\em Brou\'e invariant} of the $p$-permutation equivalence $\gamma$.  

\smallskip
If $\gamma\in T^\Delta(B,C)$ is a $p$-permutation equivalence  then, by Theorems~14.1 and 14.3 in \cite{BP2020}, there exists an indecomposable $p$-permutation $(B,C)$-bimodule $M$ with vertex of the form $\Delta(D,\phi,E)$, where $D$ is a defect group of $B$, $E$ is a defect group of $C$, and $\phi\colon E\myiso D$ is an isomorphism, and there exists a sign $\varepsilon\in\{\pm1\}$ such that
\begin{equation*}
   \gamma= \varepsilon\cdot[M] +\sum_{i=1}^r n_i\cdot [M_i]
\end{equation*}
with integers $n_1,\ldots,n_r$ and indecomposable $p$-permutation $(B,C)$-bimodules $M_1,\ldots, M_r$ with vertices that are properly contained in $\Delta(D,\phi,E)$. Thus, $M$ and $\varepsilon=\varepsilon(\gamma)$ are uniquely determined (up to isomorphism in the case of $M$) by $\gamma$. They are called the {\em maximal module} and the {\em sign} of $\gamma$.

\smallskip
In the main result of this paper we show that $\varepsilon(\gamma)\cdot\beta(\gamma)$ is determined by the \lq most local\rq\ data of $B$ and $C$, and is independent of $\gamma$. 
Let $(D,e)$ be a maximal $B$-Brauer pair. Thus, $D$ is a defect group of $B$ and $e$ is a block idempotent of $\calO C_G(D)$ with $\br_D(b)\ebar\neq 0$, where $\br_D\colon (\calO G)^D\to FC_G(D)$, $\sum_{g\in G}\alpha_g g\mapsto \sum_{g\in C_G(D)}\overline{\alpha_g} g$, is the {\em Brauer homomorphism}, a ring homomorphism, and where $(\calO G)^D$ denotes the set of $D$-fixed points of $\calO G$ under $D$-conjugation. The block algebra $FC_G(D)e$ has defect group $Z(D)$ (see \cite[Corollary~6.3.10]{Linckelmann2018}) and, up to isomorphism, it has a unique simple module $V$ (see \cite[Proposition~6.6.5]{Linckelmann2018}). 
We set
\begin{equation}\label{eqn b-invariant of B}
   b(B):=\frac{[C_G(D):Z(D)]}{\dim_F(V)}\,.
\end{equation}
Since any two maximal $B$-Brauer pairs are $G$-conjugate, the rational number $b(B)$ does not depend on the choices of $(D,e)$. Since the block algebra $\calO C_G(D)e$ has the central defect group $Z(D)$, its image under the natural epimorphism $\calO C_G(D)\to \calO [C_G(D)/Z(D)]$ is a block $B_*$ of defect zero (see \cite[Proposition~6.6.5]{Linckelmann2018}) whose unique irreducible character $\zeta$ satisfies $\zeta(1)=\dim_F(V)$ and also has defect zero.
Therefore, $b(B)$ is the {\em codegree} of the irreducible character $\zeta$ of $C_G(D)/Z(D)$ and an integer which is not divisible by $p$. We denote by 
\begin{equation*}
   \beta(B):=\overline{b(B)}\in\FF_p^\times
\end{equation*}
its residue class in $\ZZ/p\ZZ=\FF_p$.

\begin{theorem}\label{thm main}
Let $B=\calO Gb$ and $C=\calO Hc$ be block algebras as above and let $\gamma\in T^\Delta(B,C)$ be a $p$-permutation equivalence between $B$ and $C$. Then 
\begin{equation}\label{eqn main thm}
   \beta(\gamma)= \varepsilon(\gamma)\cdot\frac{\beta(B)}{\beta(C)}\,.
\end{equation}
In particular, up to a sign, $\beta(\gamma)$ is independent of $\gamma$.
\end{theorem}

\begin{remark}\label{rem Rickard}
(a) Let $X_\bullet$ be a {\em splendid Rickard equivalence} between $B$ and $C$. For our purposes this is (see \cite{Rickard1996}, where this concept was introduced first) a bounded chain complex $X_\bullet$ of finitely generated $p$-permutation $(B,C)$-bimodules, whose indecomposable direct summands have twisted diagonal vertices, such that 
\begin{equation*}
   X_\bullet\otimes_{\calO H} X_\bullet^\circ \simeq B\quad\text{and}\quad X_\bullet^\circ\otimes_{\calO G} X_\bullet\simeq C\,
\end{equation*}
where $X_\bullet^\circ$ denotes the $\calO$-dual chain complex of $X_\bullet$,
$\simeq$ means homotopy equivalence of chain complexes of $(B,B)$-bimodules (resp.~$(C,C)$-bimodules), and $B$ (resp.~$C$) denotes the chain complex with $B$ (resp.~$C$) placed in degree $0$ and being the only non-zero term. Note that our definition for the purpose of this paper is more general than the original one (see \cite{Rickard1996}) and also later definitions (see for instance \cite[Definition~9.7.5]{Linckelmann2018}). 
For $X_\bullet$ as above, the element
\begin{equation*}
   \gamma:=\sum_{n\in\ZZ}(-1)^n [X_n] \in T^\Delta(B,C)
\end{equation*}
is a $p$-permutation equivalence (see~\cite[Theorem~15.2]{BP2020}) and the element
\begin{equation*}
   \mu:=\kappa(\gamma) = \sum_{n\in\ZZ}(-1)^n \kappa([X_n]) \in R(\KK Gb,\KK Hc)
\end{equation*}
is a perfect isometry (see~\cite[Proposition~9.9]{BP2020}). This way one can define the Brou\'e invariant of $X_\bullet$ as $\beta(X_\bullet):=\beta(\gamma)=\beta(\mu)$ and the statement of Theorem~\ref{thm main} applies also to $\beta(X_\bullet)$.

\smallskip
(b) With the notation of Theorem~\ref{thm main} and the preceding paragraph, let $M_*$ be the unique indecomposable $(B_*,C_*)$-bimodule (up to isomorphism). Since $B_*$ and $C_*$ are block algebras of defect $0$, $M_*$ is a $p$-permutation bimodule, has twisted diagonal vertex, and induces a Morita equivalence between $B_*$ and $C_*$. The chain complex consisting of the only non-zero term $M_*$ in degree zero is therefore a splendid Rickard equivalence and $\gamma_*:=[M_*]\in T^\Delta(B_*,C_*)$ is a $p$-permutation equivalence between $B_*$ and $C_*$. Equation~(\ref{eqn main thm}) can now also be interpreted as $\beta(\gamma)=\epsilon(\gamma)\cdot\beta([M_*])$.
\end{remark}

Of particular interest is the situation where $C$ is the Brauer correspondent of $B$ with respect to Brauer's first main theorem.

\begin{corollary}\label{cor main}
Suppose that $D$ is a defect group of $B=\calO Gb$, that $H=N_G(D)$, and that the block idempotent $c$ of $\calO H$ is the Brauer correspondent of $b$ with respect to Brauer's first main theorem, i.e., $\bar{c}=\br_D(b)$. Then the Brou\'e invariant of any $p$-permutation equivalence and any splendid Rickard equivalence between $B$ and $C$ is equal to $1$ or $-1$. In particular, if there exists a $p$-permutation equivalence (resp.~splendid Rickard equivalence) between $B$ and $C$ then there also exists one with Brou\'e invariant equal to $1$.
\end{corollary}

\begin{proof}
In the situation of the corollary, one can choose $D=E$ so that $C_G(D)=C_H(E)$. Moreover, one can choose $f=e$ to obtain that $b(B)=b(C)$. The first statement follows now from Theorem~\ref{thm main}. The last statement follows from replacing $\gamma$ with $-\gamma$ or shifting $X_\bullet$ by one degree.
\end{proof}

Next we apply Corollary~\ref{cor main} in order to relate $p$-permutation equivalences and splendid Rickard equivalences to various refinements of the Alperin-McKay conjecture. Suppose that $b$, $D$, $H$, and $c$ are as in Corollary~\ref{cor main}. Let $\Irr_0(KGb)$ denote the set of characters of irreducible $KGb$-modules of height zero. Then the Alperin-McKay conjecture, see Conjecture~3 in \cite{Alperin1976}, states that
\begin{equation}\label{eqn AMcK}
   |\Irr_0(KGb)| = |\Irr_0(KHc)|\,.
\end{equation}
We first recall two refinements of this conjecture, introduced by Isaacs and Navarro in \cite{IN} and Navarro in \cite{Navarro2004}. For $\chi\in\Irr(KGb)$, set $r(\chi):=(|G|/\chi(1))_{p'}$, the $p'$-part of the {\em codegree} of $\chi$. For $r\in\{1,\ldots,p-1\}$, let $\Irr_0(KGb,r)$ denote the set of all $\chi\in\Irr_0(KGb)$ with $r(\chi)\equiv \pm r\mod p$. In the same way we define $\Irr_0(KHc,r)$. Conjecture~B in \cite{IN} states that
\begin{equation}\label{eqn INB}
   |\Irr_0(KGb,r)|=|\Irr_0(KHc,r)|
\end{equation}
for every $r\in\{1,\ldots,p-1\}$. Isaacs and Navarro also considered Galois actions on characters. Let $\QQ_p\subseteq L\subseteq K$ be an intermediate field and consider the set $\Irr_0(KGb, L)$ consisting of those $\chi\in\Irr_0(KGb,r)$ with $\QQ_p(\chi)=L$. Here, $\QQ_p(\chi):=\QQ_p(\chi(g)\mid g\in G)$. Similarly, define $\Irr_0(KHc,L)$. Clearly, $\Irr_0(KGb,L)$ and $\Irr_0(KHc,L)$ are empty unless $L\subseteq \QQ(\zeta)$, where $\zeta\in K$ has order $\exp(G)$. A slightly stronger version of Conjecture~B in \cite{Navarro2004} states that
\begin{equation}\label{eqn NB}
   |\Irr_0(KGb,L)|=|\Irr_0(KHc,L)|
\end{equation}
for every intermediate field $\QQ_p\subseteq L\subseteq\QQ_p(\zeta)$. Combining the refined Alperin-McKay conjectures in (\ref{eqn INB}) and (\ref{eqn NB}), let $\Irr_0(\KK Gb,r,L)$ be the set of all $\chi\in\Irr_0(\KK Gb)$ with $r(\chi)\equiv \pm r\mod p$ and $\QQ_p(\chi)=L$. Then one can consider the conjecture
\begin{equation}\label{eqn IND}
   |\Irr_0(KGb,r,L)|=|\Irr_0(KHc,r,L)|
\end{equation}
for every $r\in\{1,\ldots,p-1\}$ and every intermediate field $\QQ_p\subseteq L\subseteq \QQ_p(\zeta)$. 
Finally, Turull, see \cite{Turull2013}, suggested a further refinement involving endomorphism rings of simple modules. For $\chi\in\Irr(KGb)$, let $h(\chi)\in\QQ/\ZZ$ be defined as follows: Let $V$ be the unique irreducible $\QQ_p G$-module with the property that $\chi$ is a constituent of the character of $V$ and let $D:=\End_{\QQ_p G}(V)$, a central $\QQ_p(\chi)$-division algebra. Then let $h(\chi)\in\QQ/\ZZ$ be the Hasse-invariant of $D$, see \cite[Section~14]{Reiner1975} for a definition. 
For $h\in\QQ/\ZZ$, define $\Irr_0(KGb,r,L,h)$ as the set of those $\chi\in\Irr_0(KGb,r,L)$ with $h(\chi)=h$. Turull's conjecture then states that
\begin{equation}
\label{eqn T}
   |\Irr_0(KGb,r,L,h)|=|\Irr_0(KHc,r,L,h)|
\end{equation}
for all $r\in\{1,\ldots, p-1\}$, $\QQ_p\subseteq L\subseteq \QQ(\zeta)$, and $h\in\QQ/\ZZ$. Note that $h(\chi)$ is represented by an element in $\{0,1/p,2/p,\ldots,(p-1)/p\}$ if $p$ is odd, and by $0$ or $1/2$ if $p=2$, for any $\chi\in\Irr(KG)$, see \cite{Yamada1974a} and \cite{Yamada1974b}. 

\medskip
Corollary~\ref{cor main} allows us now to make a connection between $p$-permutation equivalences and splendid Rickard equivalences and the refinements (\ref{eqn INB}), (\ref{eqn IND}), (\ref{eqn T}) of the Alperin-McKay conjecture. Note that the definitions of $p$-permutation equivalences and splendid Rickard equivalences extend in the obvious way to arbitrary complete discrete valuation rings in place of $\calO$, in particular to $\ZZ_p$, as used in the statement of the next theorem. They also extend to sums of blocks over arbitrary complete discrete valuation rings, as used in the proof of the next theorem.

\begin{theorem}\label{thm app}
Let $b$ be a block idempotent of $\calO G$ and let $\btilde$ be the unique block idempotent of $\ZZ_p G$ with $b\btilde\neq 0$. Let $D$ be a defect group of $b$ and set $H:=N_G(D)$. Further, let $c$ be the block idempotent of $\calO H$ which corresponds to $b$ under Brauer's First Main Theorem, i.e., $\br_D(b)=\cbar$, and let $\ctilde$ be the unique block idempotent of $\ZZ_p H$ with $c\ctilde\neq 0$. Finally, let $\zeta\in K$ be a root of unity of order $\exp(G)$ and set $\Gamma:=\Gal(\QQ_p(\zeta)/\QQ_p)$.

\smallskip
{\rm (a)} If there exists a $p$-permutation equivalence between $\calO Gb$ and $\calO Hc$ then (\ref{eqn INB}) holds. In other words, there exists a bijection
\begin{equation*}
   \alpha\colon \Irr_0(KGb)\myiso \Irr_0(KHc)
\end{equation*}
such that $r(\chi)\equiv \pm r(\alpha(\chi))\!\mod p$, for all $\chi\in\Irr_0(KGb)$.

\smallskip
{\rm (b)} If there exists a $p$-permutation equivalence between $\ZZ_p G\btilde$ and $\ZZ_p H\ctilde$ then (\ref{eqn IND}) holds. 
Moreover, the $\Gamma$-stabilizers $\Gamma_b$ and $\Gamma_c$ of $b$ and $c$, respectively, coincide and there exists a $\Gamma_b$-invariant bijection
\begin{equation*}
   \alpha\colon \Irr_0(KGb)\myiso \Irr_0(KHc)
\end{equation*}
such that $r(\chi)\equiv \pm r(\alpha(\chi))\!\mod p$, for all $\chi\in\Irr_0(KGb)$.

\smallskip
{\rm (c)} If there exists a splendid Rickard equivalence between $\ZZ_p G\btilde$ and $\ZZ_p H\ctilde$ then (\ref{eqn T}) holds. 
Moreover, there exists a $\Gamma_b$-equivariant bijection
\begin{equation*}
   \alpha\colon \Irr_0(KGb)\myiso \Irr_0(KHc)
\end{equation*}
such that $r(\chi)\equiv \pm r(\alpha(\chi))\!\mod p$ and $h(\chi)=h(\alpha(\chi))$, for all $\chi\in\Irr_0(KGb)$.
\end{theorem}

\begin{remark}
(a) Splendid Rickard equivalences and $p$-permutation equivalences between block algebras of $\calO G$ with abelian defect groups and their Brauer correspondents are expected to exist according to Brou\'e's {\em abelian} defect group conjecture, but they are known not to exist for general defect groups. Therefore, Theorem~\ref{thm app} cannot be a general tool for proving the (refinements of the) Alperin-McKay Conjecture.

\smallskip
(b) For blocks with cyclic defect groups, Isaacs and Navarro proved in \cite{IN} that (\ref{eqn INB}) and a weaker version of (\ref{eqn IND}) hold. Turull proved in \cite{Turull2008} that for the same blocks a weaker version of (\ref{eqn T}) holds, which takes Schur indices  over $\QQ_p$ into account. More precisely, it was shown that there exists a bijection between $\Irr_0(\calO Gb)$ and $\Irr_0(\calO Hc)$ preserving the invariants $\QQ_p(\chi)$, and the order of $h(\chi)$ in $\QQ/\ZZ$. Turull proved in \cite{Turull2013} that a slightly weaker version of (\ref{eqn T}) holds for arbitrary blocks when $G$ is $p$-solvable. In Turull's result the invariant $h(\chi)\in\QQ/\ZZ$ is replaced by its order in $\QQ/\ZZ$, which is also the Schur index of $\chi$ over $\QQ_p$.

\smallskip
(c) Kessar and Linckelmann in \cite{KessarLinckelmann} and Huang in \cite{Huang2021} proved that there exists a splendid Rickard equivalence between blocks over $\ZZ_p$ with cyclic defect groups and Klein four group and their Brauer correspondents. Thus, by Theorem~\ref{thm app}(c), we obtain the new result that the refined version (\ref{eqn T}) of the Alperin-McKay Conjecture holds for all blocks with cyclic and Klein four defect groups. 
\end{remark}

The paper is arranged as follows. Theorem~\ref{thm main} is proved in Section~\ref{sec proof of main theorem}. The proof uses properties of $p$-permutation equivalences and the notion of extended tensor products of bimodules, a construction first introduced by Bouc in \cite{Bouc2010b}. In Section~\ref{sec bimodules} we recall this construction, prove that it can be realized as a biset operation, and prove Theorem~\ref{thm module new formula}, a formula for the extended tensor product of induced modules, which is used in the proof of Theorem~\ref{thm main}. Section~\ref{sec bimodules} uses the language of bisets which we introduce in Section~\ref{sec bisets} following \cite{Bouc2010a}. There, we also define the notion of an extended tensor product for bisets, and prove a formula for bisets, analogous to Theorem~\ref{thm module new formula} for modules. Finally, Theorem~\ref{thm app} is proved as an application of Theorem~\ref{thm main} in Section~\ref{sec application}.


\section{Bisets and extended tensor products}\label{sec bisets}

In the first part of this section we recall from \cite[Chapter~2]{Bouc2010a} the notions, notations, and results related to bisets, that we will need in Section~\ref{sec bimodules}. In the second part we introduce the notion of {\em extended tensor products}
for bisets (see \ref{noth gen biset tensor product}) and prove results about this construction. Although the content of \ref{noth gen biset tensor product}, Proposition~\ref{prop defres and gen tens}, and Theorem~\ref{thm biset new formula} will not be used in the remaining part of the paper, it is included for completeness and for future reference.

\smallskip
Throughout this section, $G$, $H$, $K$, $L$, $I$ and $J$ denote finite groups. 

\begin{nothing}\label{noth bisets} {\em $G$-sets, $(G,H)$-bisets, tensor products and external products.}\quad (a) We denote by $\lset{G}$ the category of finite left $G$-sets and by $\lset{G}_{H}$ the category of finite $(G,H)$-bisets. Recall that a $(G,H)$-biset is a set $U$ endowed with a left $G$-action and a right $H$-action that commute with each other. We will often view, without further notice, a $(G,H)$-biset $U$ as left $G\times H$-set via $(g,h)u=guh^{-1}$ for $g\in G$, $h\in H$, and $u\in U$, and vice-versa. This defines an obvious isomorphism of categories 
\begin{equation*}
  \lset{G}_H\cong \lset{G\times H}\,.
\end{equation*}

\smallskip
(b) One has a functor
\begin{equation*}
  -\otimes_H - \colon \lset{G}_H\times \lset{H}_K \to \lset{G}_K\,,
\end{equation*}
where, for $U\in\lset{G}_H$ and $V\in\lset{H}_K$, one defines $U\otimes_H V$ as the set of $H$-orbits of $U\times V$ with respect to the $H$-action given by $h\cdot(u,v):=(uh^{-1},hv)$. The $H$-orbit of $(u,v)$ is denoted by $u\otimes v\in U\otimes_H V$ (or $u\otimes_H v$ for clarity). Thus, $uh\otimes v=u\otimes hv$ for $h\in H$ and $(u,v)\in U\times V$. The set $U\otimes_H V$ has a well-defined $(G,K)$-biset structure given by $g(u\otimes v)k:=(gu)\otimes(vk)$. If $\phi\colon U\to U'$ and $\psi\colon V\to V'$ are morphisms of $(G,H)$-bisets and $(H,K)$-bisets, respectively, then $\phi\otimes_H\psi$ is defined by $(\phi\otimes_H\psi)(u\otimes v):=\phi(u)\otimes \psi(v)$. Note that in \cite[Definition~2.3.11]{Bouc2010a} this construction is denoted by $U\times_H V$ and called the {\em composition} of $U$ and $V$. Since we will use fiber products later (see \ref{noth defres}) with conflicting notation, we chose here the notation $U\otimes_H V$ and call it the {\em tensor product} of $U$ and $V$.

\smallskip
(c) Choosing $K=\{1\}$ in (b), fixing $U\in\lset{G}_H$, and using the obvious category isomorphisms $\lset{H}_{\{1\}}\cong \lset{H}$ and $\lset{G}_{\{1\}}\cong \lset{G}$, we obtain a functor
\begin{equation*}
  U\otimes_H -\colon \lset{H}\to\lset{G}\,.
\end{equation*}

\smallskip
(d) Finally, one has a functor
\begin{equation*}
  - \times - \colon \lset{G} \times \lset{H} \to \lset{G\times H}
\end{equation*}
which maps an object $(U,V)$ of $\lset{G}\times\lset{H}$ to $U\times V$ endowed with the $G\times H$-action $(g,h)(u,v):=(gu,hv)$, for $(g,h)\in G\times H$ and $(u,v)\in U\times V$.
\end{nothing}

Note that the disjoint union $U\coprod U'$ of two finite sets provides a coproduct in the categories $\lset{G}$ and $\lset{G}_H$.  The functors defined in~\ref{noth bisets} satisfy the following properties and compatibilities.

\begin{lemma}\label{lem bisets}
{\rm (a)} For $U,U'\in \lset{G}_H$ and $V,V'\in\lset{H}_K$ one has isomorphisms
\begin{equation*}
  (U\coprod U')\otimes_H V\cong (U\otimes_H V)\coprod (U'\otimes_H V)\quad \text{and}\quad
  U\otimes_H(V\coprod V') \cong (U\otimes_H V)\coprod (U\otimes_H V')
\end{equation*}
in $\lset{G}_K$, which are natural in $U$, $U'$, $V$, and $V'$.

\smallskip
{\rm (b)} For $U\in\lset{G}_H$, $V\in\lset{H}_K$, and $W\in\lset{K}_L$ (resp.~$W\in\lset{K}$) one has an isomorphism
\begin{equation*}
  (U\otimes_H V)\otimes_K W \cong U\otimes_H (V\otimes_K W)
\end{equation*}
in $\lset{G}_L$ (resp.~in $\lset{G}$), which is natural in $U$, $V$ and $W$. It maps $(u\otimes v)\otimes w$ to $u\otimes (v\otimes w)$.

\smallskip
{\rm (c)} For $U\in\lset{G}_H$, $V\in \lset{K}_L$, $R\in \lset{H}_I$, and $S\in\lset{L}_J$ one has an isomorphism
\begin{equation*}
  (U\otimes_H R)\times (V\otimes_L S) \cong (U\times V)\otimes_{H\times L}(R\times  S)
\end{equation*}
in $\lset{G\times K}_{I\times J}$, which is natural in $U$, $V$, $R$, and $S$, and which maps $(u\otimes r,v\otimes s)$ to $(u,v)\otimes (r,s)$. 

\smallskip
{\rm (d)} For $U\in\lset{G}_H$ one has isomorphisms
\begin{equation*}
  G\otimes_G U \cong U\quad\text{and}\quad U\otimes_H H\cong U
\end{equation*}
in $\lset{G}_H$, which are natural in $U$ and are given by $g\otimes u\to gu$ and $u\otimes h\to uh$. Here $G$ is viewed as $(G,G)$-biset and $H$ is viewed as $(H,H)$-biset via left and right multiplication.
\end{lemma}

The bisets defined in the following example are called {\em elementary} bisets.

\begin{example}\label{ex bisets}
(a) For $H\le G$, $\Res^G_H$ (resp.~$\Ind_H^G$) denotes the $(H,G)$-biset (resp.~$(G,H)$-biset) $G$ endowed with left and right multiplication. The resulting functors
\begin{equation*}
  \Res^G_H:=\Res^G_H\otimes_G- \colon \lset{G}\to\lset{H}\quad 
  \text{and} \quad \Ind_H^G:=\Ind_H^G\otimes_H -\colon \lset{H}\to\lset{G}
\end{equation*}
are denoted by the same symbols and are called {\em restriction} from $G$ to $H$ and {\em induction} from $H$ to $G$. Note that $\Res^G_H\cong (H\times G)/\Delta(H)$ in $\lset{H}_G\cong\lset{H\times G}$ and $\Ind_H^G\cong (G\times H)/\Delta(H)$ in $\lset{G}_H\cong \lset{G\times H}$, where $\Delta(H):=\{(h,h)\in h\in H\}$ is the stabilizer of $1\in G$.

\smallskip
(b) For $N\trianglelefteq G$, $\Inf_{G/N}^G$ (resp.~$\Def_{G/N}^G$) denotes the $(G,G/N)$-biset (resp.~$(G/N,G)$-biset) $G/N$ endowed with left and right multiplication, after using the natural epimorphism $\pi\colon G\to G/N$. The resulting functors
\begin{equation*}
  \Inf_{G/N}^G:=\Inf_{G/N}^G\otimes_{G/N}-\colon \lset{G/N}\to\lset{G} \quad \text{and} \quad 
  \Def_{G/N}^G:=\Def_{G/N}^G\otimes_G -\colon \lset{G}\to\lset{G/N}
\end{equation*}
are called {\em inflation} from $G/N$ to $G$ (resp.~{\em deflation} from $G$ to $G/N$). Note that $\Inf_{G/N}^G\cong(G\times(G/N))/\Delta(G,\pi)$ in $\lset{G}_{G/N}$ and $\Def_{G/N}^G\cong ((G/N)\times G)/\Delta(\pi,G)$ in $\lset{G/N}_G$, where 
\begin{equation*}
  \Delta(G,\pi):=\{(g,gN)\mid g\in G\}\quad\text{and}\quad \Delta(\pi,G):=\{(gN,g)\mid g\in G\}
\end{equation*}
are the stabilizers of $1\in G/N$.

\smallskip
(c) If $\alpha\colon G\myiso G'$ is an isomorphism we denote by $\Isom_\alpha$ the $(G',G)$-biset $G'$ with $G'$ acting by multiplication from the left and $G$ acting via $\alpha$ and multiplication from the right. Note that $\Isom_\alpha\cong(G'\times G)/\Delta(\alpha,G)$ in $\lset{G'}_G$, with $\Delta(\alpha,G)=\{(\alpha(g),g)\mid g\in G\}$ being the stabilizer of $1\in G'$ of the corresponding $(G'\times G)$-action. Note that if also $\beta\colon G'\myiso G''$ is an isomorphism then $\Isom_{\beta\alpha}\cong \Isom_{\beta}\otimes_{G'}\Isom_\alpha$ in $\lset{G''}_G$. 

If $\alpha= c_g\colon H\to gHg^{-1}$, $h\mapsto ghg^{-1}$, is the conjugation map for a subgroup $H\le G$ and $g\in G$, we set $\Con_g:=\Isom_{c_g}   \in \lset{gHg^{-1}}_H$ and $\lexp{g}{U}:=\Con_g(U)$, for $U\in\lset{H}$.
\end{example}

By Lemma~\ref{lem bisets}(d), the functors $\Res^G_H$, $\Inf_{G/N}^G$, $\Isom_\alpha$, and $\Con_g$ above are naturally isomorphic to the functors that don't change the underlying set, but restrict the action along the group homomorphisms $H\to G$, $G\to G/N$, $\alpha^{-1}\colon G'\to G$ and $c_g^{-1}\colon gHg^{-1}\to H$, respectively.

\begin{nothing}\label{noth subgroups of direct products} {\em Subgroups of direct products.}\quad 
Let $X\le G\times H$. We denote by $p_1\colon G\times H\to G$ and $p_2\colon G\times H\to H$ the projection maps. If one sets
\begin{equation*}
  k_1(X):=\{g\in G\mid (g,1)\in X\}\quad\text{and}\quad k_2(X):=\{h\in H\mid (1,h)\in X\}
\end{equation*}
then 
\begin{equation*}
  k_1(X)\trianglelefteq p_1(X)\le G\,,\quad k_2(X)\trianglelefteq p_2(X)\le H\,,\quad\text{and}\quad
  k_1(X)\times k_2(X)\le X\le p_1(X)\times p_2(X)\,.
  \end{equation*}
Moreover, if additionally $Y\le H\times K$, we set
\begin{equation*}
  X*Y:=\{(g,k)\in G\times K\mid \exists h\in H\colon (g,h)\in X \text{ and } (h,k)\in Y\}\,.
\end{equation*}
It is easy to see that $X*Y$ is a subgroup of $G\times K$ and that the construction $-*-$ is associative and monotonous with respect to inclusion in each argument. Moreover, one has
\begin{gather*}
  k_1(X)\le k_1(X*Y)\trianglelefteq p_1(X*Y)\le p_1(X)\,,\quad
  k_2(Y)\le k_2(X*Y)\trianglelefteq p_2(X*Y)\le p_2(Y)\,,\\
  \quad\text{and}\quad
  k_1(X)\times k_2(Y)\le X*Y\le p_1(X)\times p_2(Y)\,.
\end{gather*}
\end{nothing}

\bigskip
The following theorem is an explicit formula for the functor in \ref{noth bisets}(b) applied to transitive bisets, see~\cite[Lemma~2.3.24]{Bouc2010a}.

\begin{theorem}\label{thm Mackey formula}
Let $G$, $H$, $K$ be finite groups and let $X\le G\times H$ and $Y\le H\times K$. Then
\begin{equation*}
  (G\times H)/X \otimes_H (H\times K)/Y \cong 
  \coprod_{h\in [p_2(X)\backslash H/p_1(Y)]} (G\times K)/(X*\lexp{(h,1)}{Y})\
\end{equation*}
in $\lset{G}_K$, where $[p_2(X)\backslash H/p_1(Y)]$ denotes a set of representatives of the $(p_2(X),p_1(Y))$-double cosets in $H$. The isomorphism maps the element $(g,k)(X\**\lexp{(h,1)}{Y})$ in the $h$-component of the right hand side to $(g,1)X\otimes (h,k)Y$.
\end{theorem}

\begin{nothing}\label{noth gen biset tensor product} {\em Extended tensor products for bisets.}\quad We generalize the tensor product of bisets from \ref{noth bisets}(b) as follows. Let $X\le G\times H$, $Y\le H\times K$, $U\in\lset{X}$, $V\in\lset{Y}$, and set 
\begin{equation*}
   k(X,Y):=k_2(X)\cap k_1(Y)\le H\,.
\end{equation*} 
We may consider $U$ and $V$ via restriction as $U\in\lset{k_1(X)}_{k(X,Y)}$ and $V\in\lset{k(X,Y)}_{k_2(Y)}$ and form the tensor product $U\otimes_{k(X,Y)} V\in \lset{k_1(X)}_{k_2(Y)}\cong \lset{k_1(X)\times k_2(Y)}$. The action of $k_1(X)\times k_2(X)$ on $U\otimes_{k(X,Y)}V$ (as defined in \ref{noth bisets}) can be extended to an action of $X*Y$ as follows. Let $(g,k)\in X*Y$, $u\in U$, and $v\in V$. Choose $h\in H$ such that $(g,h)\in X$ and $(h,k)\in Y$, and set
\begin{equation*}
  (g,k)(u\otimes v):= ((g,h)u)\otimes ((h,k)v)\,.
\end{equation*}
This definition does not depend on the choice of $h$ and defines a functor that we denote by
\begin{equation}\label{eqn biset gentens}
  -\tens{X}{Y}{k(X,Y)}- \colon \lset{X}\times \lset{Y}\to \lset{X*Y}
\end{equation}
or simply by $-\tenstop{X}{Y}{}-$. We call it the {\em extended tensor product} (with respect to $X$ and $Y$). Note that this construction coincides with the construction in \ref{noth bisets}(b) when $X=G\times H$ and $Y=H\times K$ (and $X*Y=G\times K$). The extended tensor product functor is associative. 
More precisely, if also $Z\le K\times L$ and $W\in\lset{Z}$, then $(u\otimes v)\otimes w\mapsto u\otimes (v\otimes w)$ defines an isomorphism between $(U\otimes_H V)\otimes_K W\myiso U\otimes_H(V\otimes_K W)$ in $\lset{X*Y*Z}$. 
Moreover, the extended tensor product functor respects coproducts in each argument.
\end{nothing}

\begin{nothing}\label{noth defres} {\em The biset and functor $\DefRes^{X\times Y}_{X*Y}$.}\quad
For $X\le G\times H$ and $Y\le H\times K$, consider the pull-back $X\times_H Y$ of the two projection maps $p_2\colon X\to H$ and $p_1\colon Y\to H$. Thus,
\begin{equation*}
  X\times_H Y:=\{((g,h),(\htilde,k))\in X\times Y\mid h=\htilde\}\le X\times Y\,.
\end{equation*}
Moreover, consider the surjective homomorphism $\nu\colon X\times_HY\to X*Y$, $((g,h),(h,k))\mapsto (g,k)$, with kernel $\{((1,h),(h,1))\mid h\in k(X,Y)\}$ and the resulting isomorphism $\nubar\colon (X\times_H Y)/\ker(\nu)\myiso X*Y$. We define the $(X*Y,X\times Y)$-biset
\begin{equation*}
  \DefRes^{X\times Y}_{X*Y}:=
  \Isom_{\nubar}\otimes_{(X\times_H Y)/\ker(\nu)} 
  \Def^{X\times_H Y}_{(X\times_H Y)/\ker(\nu)} \otimes_{X\times_H Y}
  \Res^{X\times Y}_{X\times_H Y}
\end{equation*}
and use the same notation for the induced functor
\begin{equation*}
  \DefRes^{X\times Y}_{X*Y}:= \DefRes^{X\times Y}_{X*Y}\otimes_{X\times Y} - \colon
  \lset{X\times Y}\to\lset{X*Y}\,.
\end{equation*}
Note that by the explicit descriptions in \ref{ex bisets} of the three factors of $\DefRes^{X\times Y}_{X*Y}$ and the formula in Theorem~\ref{thm Mackey formula}, we have an isomorphism
\begin{equation}\label{eqn defres biset iso}
  \DefRes^{X\times Y}_{X*Y}\cong ((X*Y)\times (X\times Y))/\{(\nu(z),z)\mid z\in X\times_H Y\}
\end{equation}
in $\lset{X*Y}_{X\times Y}$.
\end{nothing}

The following proposition shows that the extended tensor product functor $-\tenstop{X}{Y}-$ can be regarded as a composition of a biset operation and the functor $-\times -\colon \lset{X}\times\lset{Y}\to\lset{X\times Y}$ from \ref{noth bisets}(c) and (d).

\begin{proposition}\label{prop defres and gen tens}
Let $X\le G\times H$ and $Y\le H\times K$. The functor $-\tenstop{X}{Y}-\colon \lset{X}\times \lset{Y}\to\lset{X*Y}$  in (\ref{eqn biset gentens}) is naturally isomorphic to the functor $\DefRes^{X\times Y}_{X*Y}\circ (-\times -)\colon \lset{X}\times \lset{Y}\to\lset{X*Y}$.
\end{proposition}

\begin{proof}
Let $U\in\lset{X}$ and $V\in\lset{Y}$. Using the isomorphism (\ref{eqn defres biset iso}), it suffices to show that one has an isomorphism
\begin{equation}\label{eqn biset iso}
  U\tens{X}{Y}{k(X,Y)} V \cong 
  \frac{(X*Y)\times (X\times Y)}{\{(\nu(z),z)\mid z\in X\times_H Y\}}\otimes_{X\times Y} (U\times V)
\end{equation}
of $(X*Y)$-sets which is natural in $U$ and $V$. But this follows from the following statements whose straightforward but lenghty verification we leave to the reader. Mapping $u\otimes v$ to $\overline{1}\otimes (u,v)$ is a well-defined morphism $\phi$ of $(X*Y)$-sets from the left hand side in (\ref{eqn biset iso}) to the right hand side, which is natural in $U$ and $V$. Moreover, mapping $\overline{((g,k),(x,y))}\otimes(u,v)$ to $(g,h)x^{-1}u\otimes (h,k)y^{-1}v$, where $h\in H$ is chosen such that $(g,h)\in X$ and $(h,k)\in Y$, yields a well-defined function $\psi$ from the right hand side of (\ref{eqn biset iso}) to the left hand side, such that $\phi\circ \psi$ and $\psi\circ\phi$ are the respective identity maps.
\end{proof}

\begin{lemma}\label{lem defres ind}
Let $X'\le X\le G\times H$ and $Y'\le Y\le H\times K$. Then
\begin{equation*}
  \DefRes^{X\times Y}_{X*Y}\otimes_{X\times Y}\Ind_{X'\times Y'}^{X\times Y} \cong
  \coprod_{(x,y)} 
  \Ind_{\lexp{x}{X'}*\lexp{y}{Y'}}^{X*Y} \otimes_{\lexp{x}{X'}*\lexp{y}{Y'}} 
  \DefRes^{\lexp{x}{X'}\times\lexp{y}{Y'}}_{\lexp{x}{X'}*\lexp{y}{Y'}} \otimes_{\lexp{x}{X'}\times\lexp{y}{Y'}} \Con_{(x,y)}
\end{equation*}
as $(X*Y,X'\times Y')$-bisets, where $(x,y)$ runs through a set of representatives of the $(X\times_H Y, X'\times Y')$-double cosets of $X\times Y$.
\end{lemma}

\begin{proof}
By the definition of $\DefRes^{X\times Y}_{X*Y}$ we have (omitting the indices of tensor products)
\begin{equation}\label{eqn defres ind}
   \DefRes^{X\times Y}_{X*Y}\otimes\Ind_{X'\times Y'}^{X\times Y} = \Isom_{\nubar}\otimes \Def^{X\times_H Y}_{(X\times _H Y)/\ker(\nu)} \otimes \Res^{X\times Y}_{X\times_H Y}\otimes \Ind_{X'\times Y'}^{X\times Y}\,.
\end{equation}
Using the commutation rule for $\Res$ and $\Ind$, see \cite[1.1.3.2.c]{Bouc2010a}, we find that the right hand side of (\ref{eqn defres ind}) is the coproduct of the $(X*Y,X'\times Y')$-bisets
\begin{equation}\label{eqn Lxy}
   L_{(x,y)}:= \Isom_{\nubar}\otimes \Def^{X\times_H Y}_{(X\times _H Y)/\ker(\nu)} \otimes
   \Ind_{\lexp{x}{X'}\times_H\lexp{y}{Y'}}^{X\times_H Y} \otimes
   \Res^{\lexp{x}{X'}\times\lexp{y}{Y'}}_{\lexp{x}{X'}\times_H\lexp{y}{Y'}} \otimes \Con_{(x,y)}\,,
\end{equation}
where $(x,y)$ runs through a set of representatives of the $(X*Y,X'\times Y')$-double cosets of $X\times Y$. Here we used that $(X\times_H Y)\cap(\lexp{x}{X'}\times\lexp{y}{Y'}) = \lexp{x}{X'}\times_H \lexp{y}{Y'}$. Next we use the commutation rule for $\Def$ and $\Ind$, see \cite[1.1.3.2.e]{Bouc2010a}, to obtain
\begin{equation*}
   L_{(x,y)} \cong \Isom_{\nubar}\otimes
   \Ind_{(\lexp{x}{X'} \times_H \lexp{y}{Y'})\ker(\nu)/\ker(\nu)}^{(X\times_H Y)/\ker(\nu)} \otimes \Isom _\gamma \otimes
   \Def^{\lexp{x}{X'} \times_H \lexp{y}{Y'}}_{(\lexp{x}{X'} \times_H \lexp{y}{Y'})/\ker(\nu')} \otimes 
   \Res^{\lexp{x}{X'}\times\lexp{y}{Y'}}_{\lexp{x}{X}\times_H\lexp{y}{Y'}} \otimes \Con_{(x,y)}\,,
\end{equation*}
where $\nu'\colon\lexp{x}{X'}\times_H\lexp{y}{Y'}\to \lexp{x}{X'}*\lexp{y}{Y'}$ is the epimorphism analogous to $\nu$ in \ref{noth defres} and $\gamma\colon  (\lexp{x}{X'} \times_H \lexp{y}{Y'})/\ker(\nu') \myiso (\lexp{x}{X'} \times_H \lexp{y}{Y'})\ker(\nu)/\ker(\nu)$ is the canonical isomorphism, noting that $\ker(\nu')=(\lexp{x}{X'}\times_H\lexp{y}{Y'})\cap\ker(\nu)$. Finally, the commutation of the two left most factors $\Isom$ and $\Ind$, see \cite[1.1.3.2.a]{Bouc2010a}, yields
\begin{equation*}
   L_{(x,y)} \cong \Ind_{\lexp{x}{X'}*\lexp{y}{Y'}}^{X*Y} \otimes \Isom_\delta \otimes \Isom_\gamma \otimes
   \Def^{\lexp{x}{X'} \times_H \lexp{y}{Y'}}_{(\lexp{x}{X'} \times_H \lexp{y}{Y'})/\ker(\nu')} \otimes 
   \Res^{\lexp{x}{X'}\times\lexp{y}{Y'}}_{\lexp{x}{X}\times_H\lexp{y}{Y'}} \otimes \Con_{(x,y)}\,,
\end{equation*}
with an isomorphism $\delta\colon (\lexp{x}{X'} \times_H \lexp{y}{Y'})\ker(\nu)/\ker(\nu)\myiso \lexp{x}{X'}*\lexp{y}{Y'}$ with the property that $\delta\circ \gamma=\overline{\nu'}$. Substituting the definition of $\DefRes^{\lexp{x}{X'}\times\lexp{y}{Y'}}_{\lexp{x}{X'}*\lexp{y}{Y'}}$, the proof is now complete.
\end{proof}

We can now generalize the formula in Theorem~\ref{thm Mackey formula}.

\begin{theorem}\label{thm biset new formula} 
Let $X'\le X\le G\times H$, $Y'\le Y\le H\times K$, $U\in \lset{X'}$ and $V\in \lset{Y'}$. Then one has an isomorphism of $X*Y$-sets,
\begin{equation}\label{eqn biset new formula}
  \Ind_{X'}^X(U)\tenstop{X}{Y} \Ind_{Y'}^Y(V) \cong
  \coprod_{(x,y)}
  \Ind_{\lexp{x}{X'}*\lexp{y}{Y'}}^{X*Y}
  \Bigl(\bigl(\lexp{x}{U}\bigr) \tenstop{\lexp{x}{X'}}{\lexp{y}{Y'}} \bigl(\lexp{y}{V}\bigr)\Bigr)\,,
\end{equation}
where $(x,y)$ runs through a set of representatives of the $(X\times_H Y, X'\times Y')$-double cosets of $X\times Y$. The isomorphism is induced by the maps $u\otimes v\mapsto (x\otimes u)\otimes(y\otimes v)$ from $\bigl(\lexp{x}{U}\bigr) \tenstop{\lexp{x}{X'}}{\lexp{y}{Y'}} \bigl(\lexp{y}{V}\bigr)$ to $\Ind_{X'}^X(U)\tenstop{X}{Y} \Ind_{Y'}^Y(V)$. Here we view $\lexp{x}{U}$ as the set $U$ with the left $\lexp{x}{X'}$-action via the isomorphism $c_{x^{-1}}$ (see the paragraph after Example~\ref{ex bisets}(c)). It is natural in $U$ and $V$, providing a natural isomorphism of functors $\lset{X'}\times\lset{Y'}\to\lset{X*Y}$.
\end{theorem}

\begin{proof}
By Proposition~\ref{prop defres and gen tens}, the left hand side of (\ref{eqn biset new formula}) is isomorphic to 
\begin{equation*}
  \DefRes^{X\times Y}_{X*Y}\otimes_{X\times Y} 
   \bigl(\Ind_{X'}^X(U)\times \Ind_{Y'}^Y(V)\bigr)\,,
\end{equation*}
with 
\begin{equation*}
   \Ind_{X'}^X(U)\times \Ind_{Y'}^Y(V) \cong \Ind_{X'\times Y'}^{X\times Y}(U\times V)
\end{equation*}
by Lemma~\ref{lem bisets}(c). Lemma~\ref{lem defres ind} now yields an isomorphism as in (\ref{eqn biset new formula}). Using the isomorphisms from Proposition~\ref{prop defres and gen tens}, Lemma~\ref{lem bisets}(c), and the proof of Lemma~\ref{lem defres ind}, one obtains the indicated map in the theorem.
\end{proof}


\section{Bimodules and extended tensor products}\label{sec bimodules}

In this section we recall the construction of extended tensor products of modules for group algebras (which is analogous to the construction in \ref{noth gen biset tensor product} for bisets).  It was first introduced by Bouc in \cite{Bouc2010b}. A list of properties of this construction can be found in \cite[Section~6]{BP2020}. It turns out that this construction for modules can again be viewed as a \lq biset operation\rq, see Proposition~\ref{prop defres and gen tens for modules}. This allows to derive Theorem~\ref{thm module new formula} for modules over group algebras in analogy to Theorem~\ref{thm biset new formula} for sets with group actions. Theorem~\ref{thm module new formula}  will be used in the proof of Theorem~\ref{thm main} in Section~\ref{sec proof of main theorem}.

\smallskip
Throughout this section, $G$, $H$, $K$, and $L$ denote finite groups and $\kk$ denotes a commutative ring. As with bisets, without further notice we view a $(\kk G,\kk H)$-bimodule as a left $\kk[G\times H]$-module via the obvious category isomorphism $\lmod{\kk G}_{\kk H}\cong \lmod{\kk[G\times H]}$.

\begin{nothing}\label{noth gen module tensor product}
Let $X\le G\times H$, $Y\le H\times K$, $M\in\lmod{\kk X}$ and $N\in\lmod{\kk Y}$. After restriction, $M$ can be viewed as $(\kk k_1(X),\kk k(X,Y))$-bimodule and $N$ can be viewed as $(\kk k(X,Y),\kk k_2(Y))$-bimodule, so that one obtains in the usual way a $(\kk k_1(X),\kk k_2(Y))$-bimodule $M\otimes_{\kk k(X,Y)} N$. The corresponding $\kk[k_1(X)\times k_2(Y)]$-module structure can be extended to $\kk[X*Y]$ by setting $(g,k)(m\otimes n):=(g,h)m\otimes (h,k)n$, for $(g,k)\in X*Y$ and $h\in H$ such that $(g,h)\in X$ and $(h,k)\in Y$. We call this $\kk[X*Y]$-module the {\em extended tensor product} of $M$ and $N$. This defines a functor
\begin{equation*}
   - \tens{X}{Y}{\kk k(X,Y)}- \colon \lmod{\kk X}\times \lmod{\kk Y} \to \lmod{\kk[X*Y]}
\end{equation*}
which we sometimes simply denote by $-\tenstop{X}{Y}-$. This functor is associative and respects direct sums in each argument. 
\end{nothing}

The extended tensor product behaves well under scalar extension. The proof of the following Lemma is straightforward and left to the reader.

\begin{lemma}\label{lem scalar extension}
Let $X\le G\times H$ and $Y\le H\times K$, and let $M\in\lmod{\kk X}$ and $N\in\lmod{\kk Y}$.

\smallskip
{\rm (a)} If $N$ is $\kk$-projective and $M$ is projective as right $\kk k_2(X)$-module then $M\tens{X}{Y}{\kk k(X,Y)}N$ is $\kk$-projective.

\smallskip
{\rm (b)} Assume that $\kk\to\kk'$ is a homomorphism of commutative rings. One has an isomorphism
\begin{equation*}
   \kk' \otimes_{\kk}(M \tens{X}{Y}{\kk k(X,Y)} N)\myiso (\kk'\otimes_{\kk}M)\tens{X}{Y}{\kk' k(X,Y)} (\kk'\otimes_{\kk}N)
\end{equation*}
which is functorial in $M$ and $N$. 
\end{lemma}

Note that one has obvious {\em linearization functors} $\kk-\colon \lset{G}_H\to \lmod{\kk G}_{\kk H}$ and $\kk\colon \lset{G}\to\lmod{\kk G}$, were $\kk U$ denotes the free $\kk$-module with basis $U$, for any finite set $U$. Moreover, one has a functor
\begin{equation*}
   -\cdot_H -\colon \lset{G}_H\times \lmod{\kk H} \to \lmod{\kk G}\,,\quad (U,M)\mapsto \kk U\otimes_{\kk H}M\,.
\end{equation*}
We will sometimes just write $U\cdot M$ instead of $U\cdot_H M$ to simplify the notation.
In the following lemma we compile a list of basic properties of this functor, whose straightforward verification we leave to the reader. Note that all isomorphisms in the following Lemma are natural in every variable. Recall that for $M\in\lmod{\kk G}_{\kk H}$ (resp.~$M\in\lmod{\kk G}$) and $N\in\lmod{\kk K}_{\kk L}$ (resp.~$N\in\lmod{\kk K}$) we may view $M\otimes_{\kk} N$ as $(\kk[G\times K],\kk[H\times L])$-bimodule (resp.~left $\kk[G\times K]$-module), often referred to as the {\em external} product structure.

\begin{lemma}\label{lem omnibus}
{\rm (a)} For $U\in\lset{G}_H$, $V\in\lset{H}_K$ one has $\kk(U\otimes_HV)\cong \kk U\otimes_{\kk H} \kk V$ in $\lmod{\kk G}_{\kk K}$.

\smallskip
{\rm (b)} For $U\in\lset{G}_H$, $V\in\lset {K}_L$, one has $\kk U\otimes_{\kk} \kk V\cong \kk(U\times V)$ in $\lmod{\kk[G\times K]}_{\kk[H\times L]}$.

\smallskip
{\rm (c)} For $U,U'\in\lset{G}_H$ and $M\in\lmod{\kk H}$ one has $(U\coprod U')\cdot_H M\cong (U\cdot_H M)\oplus (U'\cdot_H M)$ in $\lmod{\kk G}$.

\smallskip
{\rm (d)} For $U\in \lset{G}_H$ and $M,M'\in\lmod{\kk H}$ one has $U\cdot_H (M\oplus M')\cong (U\cdot_H M)\oplus (U\cdot_H M')$ in $\lmod{\kk G}$.

\smallskip
{\rm (e)} For $U\in\lset{G}_H$, $V\in\lset{H}_K$, $M\in\lmod{\kk K}$ one has $(U\otimes_H V)\cdot_K M \cong  U\cdot_H(V\cdot_K M)$ in $\lmod{\kk G}$.

\smallskip
{\rm (f)} For $U\in\lset{G}_H$, $V\in \lset{K}_L$, $M\in \lmod{\kk H}$, and $N\in\lmod{\kk L}$ one has an isomorphism
\begin{equation*}
  (U\cdot_H M)\otimes_{\kk} (V\cdot_L N) \cong (U\times V)\cdot_{H\times L}(M\otimes_{\kk} N)
\end{equation*}
in $\lmod{\kk[G\times K]}$.
\end{lemma}

\begin{proposition}\label{prop defres and gen tens for modules}
Let $X\le G\times H$ and $Y\le H\times K$. The functors $-\tenstop{X}{Y}-$ and $\DefRes^{X\times Y}_{X*Y}\cdot_{X\times Y} (-\otimes_{\kk}-)$ from $\lmod{\kk X}\times\lmod{\kk Y}$ to $\lmod{\kk[X*Y]}$ are naturally isomorphic.
\end{proposition}

\begin{proof}
This mirrors the proof of Proposition~\ref{prop defres and gen tens}. Let $M\in\lmod{\kk X}$ and $N\in\lmod{\kk Y}$, then by the isomorphism (\ref{eqn defres biset iso}) it suffices to show that one has an isomorphism
\begin{equation*}
  M\tens{X}{Y}{\kk k(X,Y)} N \cong 
  \frac{(X*Y)\times (X\times Y)}{\{(\nu(z),z)\mid z\in X\times_H Y\}}\cdot_{X\times Y} (M\otimes_{\kk} N)
\end{equation*}
of $\kk[(X*Y)\times(X\times Y)]$-modules. Mapping $m\otimes n$ to $\overline{1}\otimes (m\otimes n)$ defines such an isomorphism with inverse analogous to the inverse in the proof of Proposition~\ref{prop defres and gen tens}.
\end{proof}

\begin{theorem}\label{thm module new formula}
Let $X'\le X\le G\times H$, $Y'\le Y\le H\times K$, $M\in\lmod{\kk X'}$, and $N\in\lmod{\kk Y'}$. 
Then one has an isomorphism
\begin{equation}\label{eqn module new formula}
   \Ind_{X'}^X(M) \tenstop{X}{Y} \Ind_{Y'}^Y(N) \ \cong \ \bigoplus_{(x,y)} \
   \Ind_{\lexp{x}{X'}*\lexp{y}{Y'}}^{X*Y}
   \Bigl(\bigl(\lexp{x}{M}\bigr) \tenstop{\lexp{x}{X'}}{\lexp{y}{Y'}} \bigl(\lexp{y}{N}\bigr)\Bigr)\,,
\end{equation}
of $\kk[X*Y]$-modules, where $(x,y)$ runs through a set of representatives of the $(X\times_H Y, X'\times Y')$-double cosets of $X\times Y$. The isomorphism is induced by the maps $m\otimes n\mapsto (x\otimes m)\otimes(y\otimes n)$ from $\bigl(\lexp{x}{M}\bigr) \tenstop{\lexp{x}{X'}}{\lexp{y}{Y'}} \bigl(\lexp{y}{N}\bigr)$ to $\Ind_{X'}^X(M)\tenstop{X}{Y} \Ind_{Y'}^Y(N)$. Here we view $\lexp{x}M$ as the $\kk$-module $M$ endowed with the $\kk[\lexp{x}{X'}]$-module structure using the conjugation map $c_{x^{-1}}$. It is natural in $M$ and $N$, providing a natural isomorphism of functors $\lmod{\kk X'}\times\lmod{\kk Y'}\to\lmod{\kk[X*Y]}$.
\end{theorem}

\begin{proof}
Let $M\in\lmod{\kk X'}$ and $N\in\lmod{\kk Y'}$. Then
\begin{equation*}
    \Ind_{X'}^X(M) \tenstop{X}{Y} \Ind_{Y'}^Y(N) \ \cong \ \bigl((\Ind_{X'}^X)\cdot_{X'} M\bigr) \tenstop{X}{Y}
    \bigl((\Ind_{Y'}^Y) \cdot_{Y'} N\bigr)
\end{equation*}
and, by Propsition~\ref{prop defres and gen tens for modules} and Lemma~\ref{lem omnibus}(f) and (e), the latter is isomorphic to
\begin{align*}
  & \DefRes^{X\times Y}_{X*Y}\cdot_{X\times Y} \bigl((\Ind_{X'}^X\cdot_{X'} M)\otimes_{\kk} (\Ind_{Y'}^Y \cdot_{Y'} N)\bigr)\\ 
  \cong & \ 
  \DefRes^{X\times Y}_{X*Y}\cdot_{X\times Y} \bigl( (\Ind_{X'}^X\times\Ind_{Y'}^Y)\cdot_{X'\times Y'} (M\otimes_{\kk} N)\bigr)\\ 
  \cong & \ 
  \DefRes^{X\times Y}_{X*Y}\cdot_{X\times Y} \bigl( \Ind_{X'\times Y'}^{X\times Y} \cdot_{X'\times Y'} (M\otimes_{\kk} N)\bigr)\\ 
  \cong  & \   \bigl(\DefRes^{X\times Y}_{X*Y} \otimes_{X\times Y} \Ind_{X'\times Y'}^{X\times Y}\bigr) \cdot_{X'\times Y'}
     (M\otimes_{\kk} N)\,.
\end{align*}
Applying Lemma~\ref{lem defres ind}, Lemma~\ref{lem omnibus}(c) and  (e), and Proposition~\ref{prop defres and gen tens for modules}, the latter becomes isomorphic to the right hand side of (\ref{eqn module new formula}), since $\lexp{(x,y)}{(M\otimes_{\kk}N)}\cong \lexp{x}{M}\otimes_{\kk} \lexp{y}{N}$.
\end{proof}

As a special case of the above theorem with $X=G\times H$ and $Y=H\times K$ we recover Bouc's formula from \cite{Bouc2010b}. In fact if $h$ runs through a set of representatives of the $(p_2(X'),p_1(Y'))$-double cosets of $H$ then $((1,1),(h,1))$ runs through a set of representatives of the $((G\times H)\times_H(H\times K), (X'\times Y'))$-double cosets of $(G\times H)\times (H\times K)$. After renaming $X'$ and $Y'$ as $X$ and $Y$ we obtain the following formulation.

\begin{corollary}\label{cor Bouc module formula}
Let $X\le G\times H$, $Y\le H\times K$, $M\in\lmod{\kk X}$, and $N\in\lmod{\kk Y}$. Then one has an isomorphism
\begin{equation*}
   \Ind_X^{G\times H}(M)\otimes_{\kk H} \Ind_Y^{H\times K}(N) \cong
   \bigoplus_{h\in[p_2(X)\backslash H/p_1(Y)]} 
   \Ind_{X*\lexp{(h,1)}{Y}}^{G\times K} \bigl(M\tenstop{X}{\lexp{(h,1)}{Y}} \lexp{(h,1)}{N}\bigr)
\end{equation*}
of $(\kk G,\kk K)$-bimodules.
\end{corollary}

\begin{remark}
The formal nature and similarity of the proofs of Theorem~\ref{thm biset new formula} (resp.~Theorem~\ref{thm module new formula}) from Lemma~\ref{lem defres ind} is a consequence of the existence of a functor from the obvious monoidal $2$-category (replacing the usual biset category in \cite{Bouc2010a}) to the monoidal $2$-category whose objects are the categories $\lset{G}$ (resp.~$\lmod{\kk G}$). But setting up such a general framework would have been too lengthy for the purpose of this paper.
\end{remark}


\section{Proof of Theorem~\ref{thm main}}\label{sec proof of main theorem}

Before we start with the proof of Theorem~\ref{thm main} we need some preparation. Let $(\KK,\calO, F)$ be a $p$-modular system and let $X$ be a finite group. First we formulate for convenient reference the following well-known lemma.

\begin{lemma}\label{lem vertex 1}
Let $a\in Z(\calO X)$ be a block idempotent, $A:=\calO X a$ the corresponding block algebra, and let $M\in\lmod{A}$ be indecomposable with vertex $P$.

\smallskip
{\rm (a)} $P$ is contained in a defect group of $A$, and if $M$ is $\calO$-free then $\rk_{\calO}(M)$ is divisible by $[X:P]_p$.

\smallskip
{\rm (b)} Let $M'\in\lmod{\calO N_X(P)}$ be the Green correspondent of $M$ and let $(P,d)$ be a Brauer pair such that $dM'\neq \{0\}$. Then $dM'$ is an indecomposable $\calO N_{X(P,d)}d$-module with vertex $P$ and $M'\cong \Ind_{N_X(P,d)}^{N_X(P)}(dM')$.
\end{lemma}

\begin{proof}
(a) See \cite[Theorems~5.1.9(i) and 4.7.5]{NagaoTsushima1989}.

\smallskip
(b) Let $d'$ be the block idempotent of $\calO N_X(P)$ to which $M'$ belongs. Since $P$ is normal in $N_X(P)$, $d'$ is contained in $\calO C_X(P)$ (see \cite[Theorem~6.2.6(ii)]{Linckelmann2018}) and $dd'=d$, since $\{0\}\neq dM=dd'M$ implies $dd'\neq0$. 
 Let $I:=N_X(P,d)$ denote the stabilizer of $(P,d)$. Then $dM'$ is an $\calO Id$-module. The $(\calO Id,\calO N_X(P)d')$-bimodule $d\calO N_X(P)d'=d\calO N_X(P)$ and the $(\calO N_X(P)d',\calO Id)$-bimodule $d'\calO N_X(P)d=\calO N_X(P)d$ induce mutually inverse Morita equivalences between $\lmod{\calO N_X(Q)d'}$ and $\lmod{\calO Id}$ (see \cite[Theorem~6.2.6(iii)]{Linckelmann2018}). Moreover, these functors are naturally isomorphic to $d\cdot\Res^{N_X(P)}_I$ and $\Ind_I^{N_X(P)}$, respectively. Since $M'$ is indecomposable, so is its image $dM'=d\cdot\Res^{N_X(P)}_I(M')$ under the Morita equivalence. Moreover, if $Q$ is a vertex of $M'$ then $M'=\Ind_I^{N_X(P)}(dM')$ implies that $P\le Q$ and $dM'\mid\Res^{N_X(P)}_I(M')$ implies that $Q\le P$.
\end{proof}

\begin{nothing}\label{noth pperm and Brauer construction} We recall some facts about $p$-permutation-modules and the Brauer construction (see~\cite[Section~3]{BP2020} for more details).  $M\in\lmod{\calO X}$ (resp.~$M\in\lmod{FX}$) is called a {\em $p$-permutation} module if it is isomorphic to a direct summand of a permutation module. We denote the Grothendieck group of the category of $p$-permutation $\calO X$-modules with respect to split exact sequences by $T(\calO X)$. If $a\in Z(\calO X)$, we similarly define $T(\calO Xa)$, $T(FX)$, and $T(FX\bar{a})$. The functor $F\otimes_{\calO}-$ induces an isomorphism $T(\calO Xa)\myiso T(FX\bar{a})$, $[M]\mapsto [\overline{M}]$, preserving indecomposablility and vertices. The {\em Brauer construction} with respect to a $p$-subgroup $P\le X$ is a functor $-(P)\colon\lmod{\calO X}\to\lmod{FN_X(P)}$ that takes $p$-permutation $\calO Xa$-modules to $p$-permutation $FN_X(P)\br_P(a)$-modules and defines a group homomorphism $-(P)\colon T(\calO Xa)\to T(FN_G(P)\br_X(a))$. If $M$ is an indecomposable $p$-permutation $\calO X$-module with vertex $P$, then $M(P)$ and the Green correspondent $M'$ of $M$ are related via $\overline{M'}\cong M(P)$. For $\omega\in T(\calO G)$ and a Brauer pair $(P,e)$ of $\calO X$, we write (as in \cite{BP2020}) $\omegabar(P,e)\in T(F[N_X(P,e)]\ebar)$ for the element obtained by first applying $-(P)$ and then multiplying with the idempotent $\ebar$, and by $\omega(P,e)$ we denote the corresponding element in $T(\calO [N_X(P,e)]e)$. We call $(P,e)$ an {\em $\omega$-Brauer pair} if $\omega(P,e)\neq 0$ in $T(\calO [N_X(P,e)]e)$.
\end{nothing}

For the proof of Theorem~\ref{thm main} and the rest of this section, we fix again finite groups $G$ and $H$, and assume that $\calO$ has a root of unity of order $\exp(G\times H)$ as in Section~\ref{sec intro}. Furthermore, we fix block idempotents $b\in Z(\calO G)$ and $c\in Z(\calO H)$, and a $p$-permutation equivalence $\gamma\in T^\Delta(B,C)$ between the block algebras $B:=\calO Gb$ and $C:=\calO Hc$. Finally, we fix a maximal $\gamma$-Brauer pair (viewing $\gamma$ as an element in $T(\calO[G\times H])$). By \cite[Remark~10.2 and Theorem~10.11]{BP2020} it is of the form $(\Delta(D,\phi,E),e\otimes f^*)$, for a maximal $B$-Brauer pair $(D,e)$ and a maximal $C$-Brauer pair $(E,f)$. Thus, $D$ is a defect group of $B$ and $E$ is a defect group of $C$. Here, we write $-^*\colon \calO X\to \calO X$ for the map defined by $x\mapsto x^{-1}$, for $x\in X$. Note that this makes sense, since $C_{G\times H}(\Delta(D,\phi,E))= C_G(D)\times C_H(E)$. Finally, we set $I:=N_G(D,e)$ and $J:=N_H(E,f)$.

\begin{lemma}\label{lem vertex 2}
Let $M\in\lmod{\calO G}_{\calO H}$ be indecomposable with vertex $X\le \Delta(D,\phi,E)$, and let $L\in\lmod{\calO H}$ be indecomposable with vertex $Y\le E$. If $X<\Delta(D,\phi,E)$ or $Y< E$ then every indecomposable direct summand of $M\otimes_{\calO H} L\in\lmod{\calO G}$ has a vertex strictly contained in $D$. If additionally $M$ and $L$ are $\calO$-free then also $M\otimes_{\calO H}L$ is $\calO$-free and its rank is divisible by $p\cdot[G:D]_p$.
\end{lemma}

\begin{proof}
By Corollary~\ref{cor Bouc module formula}, each indecomposable direct summand $N$ of $M\otimes_{\calO H}L$ satisfies
\begin{equation*}
   N\mid 
   \Ind_{X*\lexp{h}{Y}}^G  \bigl( \Res^{G\times H}_X(M)\tens{X}{\lexp{h}{Y}}{\calO k(X,\lexp{h}{Y})} \Res^H_{\lexp{h}Y}(L) \bigr)
\end{equation*}
for some $h\in H$, since $M\mid \Ind_X^{G\times H}(\Res^{G\times H}_X(M))$ and $L\mid \Ind_Y^H(\Res^H_Y(L))$. It is straightforward to verify that if $X<\Delta(D,\phi,E)$ or $Y<E$ then $X*\lexp{h}{Y}<D$, so that $N$ has a vertex properly contained in $D$. 
If $M$ and $L$ are $\calO$-free then so is $\Res^{G\times H}_X(M)\tenstop{X}{\lexp{h}{Y}} \Res^H_{\lexp{h}Y}(L)$, since $k(X,\lexp{h}{Y})=\{1\}$. 
Thus, $N$ is $\calO$-free of $\calO$-rank divisible by $p\cdot[G:D]_p$ (see Lemma~\ref{lem vertex 1}(a)). The result now follows.
\end{proof}

We will prove Theorem~\ref{thm main} in four steps.

\bigskip
{\sc Step 1.}\quad Let $\psi\in\Irr(\KK Hc)$ be an irreducible character of height zero. Then 
\begin{equation}\label{eqn psi(1) a}
   \psi(1)_p=[H:E]_p \quad \text{and} \quad \psi(1)_{p'}=[H:E]_p^{-1}\cdot \psi(1)\,.
\end{equation}
Let $L$ be an $OHc$-lattice with character $\psi$. Then $L$ is indecomposable and Lemma~\ref{lem vertex 1}(a) implies that $E$ is a vertex of $L$. Let $L'$ be the Green correspondent of $L$. Then $\Ind_{N_H(E)}^H (L') \cong L \oplus \Ltilde$ for some $\calO H$-lattice $\Ltilde$ whose rank is divisible by $p\cdot[H:E]_p$ by Lemma~\ref{lem vertex 1}(a). Thus, with (\ref{eqn psi(1) a}) we have
\begin{equation}\label{eqn psi(1) b}
   \psi(1)_{p'} = [H:E]_p^{-1}\cdot \rk_\calO(L) \equiv [H:E]_p^{-1} \cdot \rk_\calO (\Ind_{N_H(E)}^H(L')) \mod p\,.
\end{equation}
Let $c'\in Z(\calO N_H(E))$ be the block idempotent corresponding to $c$ via Brauer's first main theorem, i.e., $\br_E(c)=\overline{c'}$. By \cite[Corollary~5.3.11]{NagaoTsushima1989}, $c'$ acts as the identity on $L'$. Since $(E,f)$ is a $C$-Brauer pair,  $\overline{fc'}=\fbar\br_e(c)\neq 0$ and therefore $fc'\neq 0$. Moreover, $c'$ is the sum of the distinct $N_H(E)$-conjugates of $f$ (see \cite[Theorem~6.2.6(iii)]{Linckelmann2018}). Thus, $c'L\neq\{0\}$ implies $fL'\neq\{0\}$. 
Now Lemma~\ref{lem vertex 1}(b) implies that $L'=f\cdot\Ind^{N_H(E)}_J (L'')$ for the indecomposable $\calO Jf$-module $L'':=fL'$. With (\ref{eqn psi(1) b}) we obtain
\begin{equation*}
   \psi(1)_{p'} \equiv [H:E]_p^{-1} \cdot \rk_\calO (\Ind_J^H(L'')) = \frac{[H:J]}{[H:E]_p} \cdot \rk_\calO(L'') = 
   \frac{[H:J]_{p'}}{[J:E]_p} \cdot \rk_\calO (L'') \mod p\,.
\end{equation*}
Since the left hand side of this congruence is not divisible by $p$, we have $\rk_\calO(L'')_p=[J:E]_p$ and
\begin{equation}\label{eqn psi(1) c}
   \psi(1)_{p'}\equiv [H:J]_{p'}\cdot \rk_{\calO}(L'')_{p'} \mod p\,.
\end{equation}

\smallskip
{\sc Step 2.}\quad Next, let $\mu:=\kappa(\gamma)\in R(\KK Gb, \KK Hc)$ be as in the introduction and set 
\begin{equation*}
   \chi:=I(\mu)(\psi)= \mu\cdotH \psi\,.
\end{equation*} 
Since $\mu$ is a perfect isometry (see \cite[Proposition~9.9]{BP2020}), we have $\chi\in \pm \Irr(\KK Gb)$. The goal of this step is to determine the congruence class of $\chi(1)_{p'}$ modulo $p$ in terms of local data. First recall that the perfect isometry $\mu$ preserves heights (see \cite[Lemme~1.6]{Broue1990}) so that 
\begin{equation}\label{eqn chi(1) a}
   \chi(1)_p= [G:D]_p \quad \text{and} \quad \chi(1)_{p'} = [G:D]_p^{-1}\cdot \chi(1) =  [G:D]_p^{-1}\cdot (\mu\cdotH \psi)(1)\,.
\end{equation}
By \cite[Theorems~14.1, 14.3, and 10.11]{BP2020} we can write
\begin{equation*}
   \gamma = \varepsilon \cdot [M] + \sum_{i=1}^r n_i \cdot [M_i]\,,
\end{equation*}
with $\varepsilon:=\varepsilon(\gamma)\in\{\pm 1\}$ the sign of $\gamma$, integers $n_1,\ldots, n_r$, an indecomposable $(B,C)$-bimodule $M$ with vertex $\Delta(D,\phi, E)$ and indecomposable $(B,C)$-bimodules $M_i$, $i=1,\ldots,r$, each of which has a vertex strictly contained in $\Delta(D,\phi,E)$.  
By Lemma~\ref{lem scalar extension}(a),  $M\otimes_{\calO H} L$ and $M_i \otimes_{\calO H} L$ are $\calO$-free with 
\begin{equation}\label{eqn chi(1) b}
   (\mu\cdotH\psi)(1) = \varepsilon\cdot\rk_\calO\bigl(M\otimes_{\calO H} L\bigr) + 
   \sum_{i=1}^r n_i\cdot\rk_\calO\bigl(M_i \otimes_{\calO H} L\bigr)\,.
\end{equation}
Moreover, since $M\otimes_{\calO H} L$ and $M_i \otimes_{\calO H} L$ are $\calO Gb$-lattices, their $\calO$-ranks are divisible by $[G:D]_p$ (see Lemma~\ref{lem vertex 1}), and by Lemma \ref{lem vertex 2}, the $\calO$-rank of $M_i \otimes_{\calO H} L$, $i=1,\ldots,r$, is divisible by $p\cdot[G:D]_p$. Therefore,
\begin{equation*}
   [G:D]_p^{-1} \cdot \rk_\calO(M_i\otimes_{\calO H}L) \equiv 0 \mod p
\end{equation*}
for all $i=1,\ldots, r$. Together with (\ref{eqn chi(1) a}) and (\ref{eqn chi(1) b}) this implies
\begin{equation}\label{eqn chi(1) c}
   \chi(1)_{p'} \equiv \varepsilon\cdot [G:D]_p^{-1}\cdot \rk_{\calO}(M\otimes_{\calO H} L) \mod p\,.
\end{equation}
With $L''\in\lmod{\calO Jf}$ and $L'\cong\Ind_J^{N_H(E)}(L'')$ as in Step~1, we can write $N_{N_H(E)}^H(L') \cong L\oplus \Ltilde$, with each indecomposable direct summand of $\Ltilde$ having a vertex $Q$ strictly contained in $E$. By Lemma~\ref{lem vertex 2}, $p\cdot[G:D]_p$ divides $\rk_\calO(M\otimes_{\calO}\Ltilde)$. Thus, with (\ref{eqn chi(1) c}) we obtain
\begin{equation}\label{eqn chi(1) d}
   \chi(1)_{p'}\equiv \varepsilon \cdot [G:D]_p^{-1}\cdot \rk_{\calO}\bigl(M\otimes_{\calO H}\Ind_J^H(L'')\bigr) \mod p\,.
\end{equation}
Set $Y':=N_{G\times H}(\Delta(D,\phi,E))$ and $Y'':=N_{G\times H}(\Delta(D,\phi,E),e\otimes f^*)\le Y'$ and let $M'\in\lmod{\calO Y'}$ be the Green correspondent of $M$. Then $F\otimes_\calO M'= M(\Delta(D,\phi,E))$, since $M$ is a $p$-permutation module with vertex $\Delta(D,\phi,E)$. Since $M_1,\ldots,M_r$ have vertices strictly contained in $\Delta(D,\phi,E)$, we have $0\neq(e\otimes f^*)\cdot\gamma(\Delta(D,\phi,E)) = [(e\otimes f^*)M']$, and Lemma~\ref{lem vertex 1}(b) implies that $M'=\Ind_{Y''}^{Y'}(M'')$ for the indecomposable $\calO Y''(e\otimes f^*)$-module $M'':=(e\otimes f^*)M'$ with vertex $\Delta(D,\phi,E)$. We have $\Ind_{Y'}^{G\times H}(M')\cong M\oplus \Mtilde$ for some $\calO [G\times H]$-lattice $\Mtilde$, each of whose indecomposable direct summands have a vertex strictly contained in $\Delta(D,\phi,E)$. Lemma~\ref{lem vertex 2} implies that $p\cdot[G:D]_p$ divides $\rk_\calO\bigl(\Mtilde\otimes_{\calO H}\Ind_J^H(L'')\bigr)$ and with (\ref{eqn chi(1) d}) we obtain
\begin{equation}\label{eqn chi(1) e}
   \chi(1)_{p'}\equiv \varepsilon \cdot [G:D]_p^{-1}\cdot 
   \rk_\calO\bigl(\Ind_{Y''}^{G\times H}(M'')\otimes_{\calO H} \Ind_J^H(L'')\bigr) \mod p\,.
\end{equation}
By \cite[Proposition~11.1]{BP2020}, we have 
\begin{equation}\label{eqn k p Y}
  p_1(Y'')=I\,,\quad p_2(Y'')=J\,,\quad k_1(Y'')=C_G(D)\,, \quad\text{and}\quad k_2(Y'')=C_H(E)\,.
\end{equation}
Since $p_2(Y'')=J$, Corollary~\ref{cor Bouc module formula} (with $K=\{1\}$) implies that 
\begin{equation}\label{eqn chi(1) f}
   \Ind_{Y''}^{G\times H}(\Mtilde)\otimes_{\calO H}\Ind_J^H(L'')\cong 
   \bigoplus_{h\in[J\backslash H/J]} \Ind_{Y''*\hJ}^G\bigl(M''\tenstop{Y''}{\hJ} \lexp{h}{L''}\bigr)\,.
\end{equation}
We study the direct summands in (\ref{eqn chi(1) f}). 

Case (i): Suppose $h\in H$ but $h\notin N_H(E)$. Let $S\in\lmod{\calO \hE}$ be a source of $\lexp{h}{L''}$. Since $M''$ has trivial source, we have
\begin{equation*}
   M''\tenstop{Y''}{\hJ} \lexp{h}{L''}\quad \mid \quad 
   \Ind_{\Delta(D,\phi,E)}^{Y''}(\calO)\tenstop{Y''}{\hJ}\Ind_{\hE}^{\hJ}(S)\,.
\end{equation*}
Moreover, Theorem~\ref{thm module new formula} (with $K=\{1\}$) and Lemma~\ref{lem scalar extension}(a) imply that the latter module is a direct sum of $\calO[Y''*\hJ]$-lattices that are induced from subgroups of the form $\lexp{(g,h_1)}{\Delta(D,\phi,E)}*\lexp{h_2h}{E}$ with $(g,h_1)\in Y''$ and $h_2\in\hJ$. Since $p_2(Y'')=J\le N_H(E)$, we obtain $p_2(\lexp{(g,h_1)}{\Delta(D,\phi,E)})=\lexp{h_1}{E}=E$. Since $\lexp{h_2h}{E}=\hE$ and since $E\cap \hE<E$ by the choice of $h$, the group $\lexp{(g,h_1)}{\Delta(D,\phi,E)}*\lexp{h_2h}{E}$ is properly contained in $D$. Thus, by Lemma~\ref{lem vertex 1}(a), we obtain
\begin{equation*}
   [G:D]_p^{-1}\cdot \rk_\calO\bigl(\Ind_{Y''*\hJ}^G\bigl(M''\tenstop{Y''}{\hJ} \lexp{h}{L''}\bigr)\bigr) \equiv 0 \mod p
\end{equation*}
in this case.

Case (ii): Suppose $h\in N_H(E)$ but $h\notin J$. We claim that in this case $M''\tenstop{Y''}{\hJ} \lexp{h}{L''}=\{0\}$. In fact, since $k_1(Y'')=C_G(D)$, $k_2(Y'')=C_H(E)\le \hJ$ (see (\ref{eqn k p Y})), we have $k(Y'',\hJ)= C_H(E)$. Thus, by the definition of $-\tenstop{Y''}{\hJ}-$, we have 
\begin{equation*}
   \Res^{Y''*\hJ}_{C_G(D)}\bigl(M''\tenstop{Y''}{\hJ}\lexp{h}{L''}\bigr) = 
   \Res^{Y''}_{C_G(D)\times C_H(E)}(M'')\otimes_{\calO C_H(E)} \Res^{\hJ}_{C_H(E)}(\lexp{h}{L''})\,.
\end{equation*}
Since the block idempotent $f$ of $\calO C_H(E)$ acts as the identity from the right on $\Res^{Y''}_{C_G(D)\times C_H(E)}(M'')=eM'f$, the block idempotent $\lexp{h}{f}$ acts as the identity from the left on $\Res^{\hJ}_{C_H(Q)}(\lexp{h}{L''})=\lexp{h}{f}\cdot L'$, and since $f\cdot\lexp{h}{f}=0$, the claim is proved.

With the conclusions for Case (i) and Case (ii), (\ref{eqn chi(1) e}) and (\ref{eqn chi(1) f}) imply that
\begin{equation*}
   \chi(1)_{p'} \equiv \varepsilon \cdot [G:D]_p^{-1} \cdot \rk_\calO \bigl(\Ind_{Y''*J}^G(M''\tenstop{Y''}{J} L'')\bigr) \mod p\,,
\end{equation*}
with $Y''*J=I$, since $p_1(Y'')=I$ and $p_2(Y'')=J$ (see (\ref{eqn k p Y})). Thus,
\begin{equation*}
   \chi(1)_{p'} \equiv \varepsilon \cdot \frac{[G:I]}{[G:D]_p} \cdot \rk_\calO(M''\tenstop{Y''}{J} L'') 
   = \varepsilon \cdot \frac{[G:I]_{p'}}{[I:D]_p} \cdot \rk_\calO(M''\tenstop{Y''}{J} L'')\mod p\,.
\end{equation*}
Since the left hand side of this congruence is not divisible by $p$, we have $\rk_\calO(M''\tenstop{Y''}{J} L'')_p=[I:D]_p$ and
\begin{equation}\label{eqn chi(1) g}
   \chi(1)_{p'} \equiv \varepsilon \cdot [G:I]_{p'} \cdot \rk_\calO(M''\tenstop{Y''}{J} L'')_{p'} \mod p\,.
\end{equation}

\smallskip
{\sc Step 3.}\quad Let $V$ be the unique simple $FC_G(D)e$-module and let $W$ be the unique simple $FC_H(E)f$-module as defined in the paragraph preceding Theorem~\ref{thm main}. We claim that
\begin{equation}\label{eqn V W ratio}
   \frac{\rk_\calO(M''\tenstop{Y''}{J} L'')}{\rk_{\calO}(L'')} = \frac{\dim_F(V)}{\dim_F(W)}\,.
\end{equation}
By~\cite[Proposition~14.4]{BP2020} with $S=C_G(D)$ and $T=C_H(E)$, the $(\calO C_G(D)e,\calO C_H(E)f)$-bimodule $M''':=\Res^{Y''}_{C_G(D)\times C_H(E)}(M'')$ induces a Morita equivalence between the block algebras $\calO C_G(D)e$ and $\calO C_H(E)f$. 
Therefore, the $(FC_G(D)\ebar,FC_H(E)\fbar)$-bimodule $\overline{M'''}:=F\otimes_{\calO}M'''$ induces a Morita equivalence between the block algebras $FC_G(D)\ebar$ and $FC_H(E)\fbar$. This implies that 
\begin{equation*}
   \overline{M'''}\otimes_{FC_H(E)} W \cong V\,,
\end{equation*}
since $V$ and $W$ are the unique simple modules in these block algebras. Moreover, the multiplicity of $W$ as composition factor in $\overline{L''}:=F\otimes_{\calO} L''$ is equal to the multiplicity of $V$ as composition factor in $\overline{M'''}\otimes_{FC_H(E)} \overline{L''}$. Thus,
\begin{equation*}
   \frac{\dim_F(\overline{M'''}\otimes_{FC_H(E)}\overline{L''})}{\dim_F(\overline{L''})} = \frac{\dim_F(V)}{\dim_F(W)}
\end{equation*}
and since $\dim_F(\overline{M'''}\otimes_{FC_H(E)}\overline{L''}) = \rk_\calO(M''\tenstop{Y''}{J} L'')$ and $\dim_F(\overline{L''}) = \rk_\calO(L'')$ by~Lemma~\ref{lem scalar extension}(b), Equation~(\ref{eqn V W ratio}) holds.

\smallskip
{\sc Step 4.}\quad Since we have an isomorphism $\phi\colon E\myiso D$, we obtain $|Z(D)|=|Z(E)|$. Moroever, since the fractions in (\ref{eqn Broue quotient}) and $b(B)$ and $b(C)$ are units in $\ZZ_{(p)}$, it suffices to show that 
\begin{equation}\label{eqn final a}
   \bigl(|G|\cdot\psi(1)\cdot |C_H(E)|\cdot\dim_F(V)\bigr)_{p'} \equiv \varepsilon\cdot \bigl(|H|\cdot\chi(1)\cdot |C_G(D)|\cdot \dim_F(W)\bigr)_{p'} \mod p\,.
\end{equation}
Using (\ref{eqn psi(1) c}) and (\ref{eqn V W ratio}), the left hand side of (\ref{eqn final a}) is congruent to
\begin{align}\label{eqn final b}
  &  |G|_{p'} \cdot [H:J]_{p'} \cdot \rk_\calO(L'')_{p'} \cdot |C_H(E)|_{p'} \cdot \dim_F(V)_{p'}\notag\\
   = \ & |G|_{p'} \cdot [H:J]_{p'} \cdot |C_H(E)|_{p'} \cdot \rk_\calO(M''\tenstop{Y''}{J} L'')_{p'} \cdot 
          \dim_F(W)_{p'}
\end{align}
modulo $p$. Using (\ref{eqn chi(1) g}), the right hand side of (\ref{eqn final a}) is congruent to
\begin{equation}\label{eqn final c}
   \varepsilon \cdot |H|_{p'} \cdot \varepsilon \cdot [G:I]_{p'} \cdot \rk_\calO(M''\tenstop{Y''}{J} L'')_{p'} 
   \cdot |C_G(D)|_{p'} \cdot \dim_F(W)_{p'}
\end{equation}
modulo $p$. But since $[I:C_G(D)]=[J:C_H(E)]$ by \cite[Proposition~11.1]{BP2020}, the integers in (\ref{eqn final b}) and (\ref{eqn final c}) are equal. This proves the congruence in (\ref{eqn final a}) and completes the proof of Theorem~\ref{thm main}.
\qed


\section{Proof of Theorem~\ref{thm app}}\label{sec application}

\begin{proof} {\em of Theorem~\ref{thm app}(a).}\quad
Let $\gamma\in T^\Delta(B,C)$ be a $p$-permutation equivalence between $B$ and $C$. Then $\mu:=\kappa_{G\times H}(\gamma)\in R(KGb,KHc)$ is a perfect isometry. Let $\alpha\colon \Irr(KHc)\myiso \Irr(KGb)$ be the bijection induced by the perfect isometry $I_\mu$. Since the quotient in (\ref{eqn Broue quotient}) is a unit in $\ZZ_{(p)}$, the bijection $\alpha$ preserves heights. Moreover, since $\beta(\mu)\in\{\pm 1\}$, we obtain $r(\alpha(\chi))\equiv \pm r(\chi)\mod p$ for all $\chi\in\Irr(KHc)$.
\end{proof}

\begin{proof} {\em of Theorem~\ref{thm app}(b).}\quad Block idempotents of $\calO G$ and $\calO H$ have coefficients in $\ZZ_p[\zeta']$, where $\zeta'$ is the $\exp(G\times H)_{p'}$-th power of $\zeta$. Moreover, the natural map between $\Gal(\QQ_p(\zeta')/\QQ_p)$ and the Galois group of the residue field of $\QQ_p(\zeta')$ over $\FF_p$ is an isomorphism. Therefore, \cite[Theorem~4.2]{BoltjeYilmaz2021} implies that $\Gamma_b=\Gamma_c$.

\smallskip
Let $\gammatilde\in T^\Delta(\ZZ_pG\btilde,\ZZ_pH\ctilde)$ with $\gammatilde\cdot_H\gammatilde^\circ=[\ZZ_pG\btilde]\in T^\Delta(\ZZ_pG\btilde,\ZZ_pG\btilde)$ and $\gammatilde^\circ\circ_G\gammatilde = [\ZZ_pH\ctilde]\in T^\Delta(\ZZ_pH\ctilde,\ZZ_pH\ctilde)$ and let $\gamma\in T^\Delta(\calO G\btilde,\calO H\ctilde)$ be the image of $\gammatilde$ under the natural map induced by scalar extension from $\ZZ_p$ to $\calO$. 
Then $\gamma\cdot_H\gamma^\circ=[\calO G\btilde]$ and $\gamma^\circ\cdot_G\gamma=[\calO H\ctilde]$.
Set $\mu:=\kappa_{G\times H}(\gamma)$ and let $\alpha\colon\Irr(KH\ctilde)\myiso\Irr(KG\btilde)$ be the bijection induced by the perfect isometry $I_\mu=\mu\cdot_H -\colon R(KH\ctilde)\to R(KG\btilde)$.

\smallskip
For fixed $r\in\{1,\ldots,p-1\}$, the bijection $\alpha$ restricts to a bijection
\begin{equation*}
   \alpha\colon \Irr_0(KH\ctilde,r)\myiso\Irr_0(KG\btilde,r)\,,
\end{equation*}
by \cite[Th\'eor\`eme~1.5(2)]{Broue1990}. Moreover, since $\mu$ is the scalar extension from $\QQ_p$ to $K$ of the element $\kappa(\gammatilde)\in R(\QQ_pG\btilde,\QQ_pH\ctilde)$, $\alpha$ is $\Gamma$-equivariant. Therefore,
\begin{equation}\label{eqn Delta 1}
   |\Irr_0(KG\btilde,r)^\Delta| = |\Irr_0(KH\ctilde,r)^\Delta|
\end{equation}
for all subgroups $\Delta\in\Gamma_b$.

Let $b_1,\cdots,b_n$ (resp. $c_1,\ldots, c_n$) denote the elements of the $\Gamma$-orbit of $b$ (resp.~c). Note that they have the same length, since $\Gamma_b=\Gamma_c$. Then $\btilde=b_1+\cdots+b_n$ and $\ctilde=c_1+\cdots+c_n$. For fixed $r\in\{1,\ldots,p-1\}$ the finite $\Gamma$-set $\Irr_0(KG\btilde,r)$ is the disjoint union of the subsets $\Irr_0(KGb_i,r)$, $i=1,\ldots,n$, which are permuted under the $\Gamma$-action. Since $\Gamma$ is abelian, for every subgroup $\Delta\in\Gamma_b$ the cardinality of the $\Delta$-fixed point sets $\Irr_0(KGb_i, r)^\Delta$, $i=1,\ldots,n$, is independent of $i$, so that 
\begin{equation}\label{eqn Delta 2}
   |\Irr_0(KG\btilde,r)^\Delta| = n\cdot |\Irr_0(KGb,r)^\Delta|\,.
\end{equation}
Similarly, one obtains
\begin{equation}\label{eqn Delta 3}
   |\Irr_0(KH\ctilde,r)^\Delta| = n\cdot |\Irr_0(KHc,r)^\Delta|\,.
\end{equation}
Now (\ref{eqn Delta 2}), (\ref{eqn Delta 3}), and (\ref{eqn Delta 1}) imply that
\begin{equation*}
   |\Irr_0(KGb,r)^\Delta| = |\Irr_0(KHc,r)^\Delta|
\end{equation*}
for every subgroup $\Delta\le \Gamma_b$. But this implies that one has an isomorphism
\begin{equation*}
   \Irr_0(KGb,r) \cong \Irr_0(KHc,r)
\end{equation*}
of $\Gamma_b$-sets. Taking the disjoint union over the values of $r$ modulo $\pm1$, the second statement of Part~(b) follows. The first statement follows from the second noting that the $\Gamma$-stabilizer of $\QQ_p(\chi)$ must be contained in $\Gamma_b$ for every $\chi\in\Irr(KGb)$. 
\end{proof}

In order to prove Part~(c) of Theorem~\ref{thm app}, we need one more proposition and need to introduce some additional notation.
Let $A$ and $B$ be algebras over a field $k$. For a left $A$-module $M$ we denote by $M^\circ:=\Hom_k(M,k)$ the $k$-dual of $A$, viewed as right $A$-module. By $K(A)$ we denote the homotopy category of bounded chain complexes of finitely generated $A$-modules. We identify the category of $(A,B)$-bimodules with the category of left $A\otimes_k B^\circ$-modules in the usual way, where $B^\circ$ denotes the opposite $k$-algebra of $B$. We denote by $K(A,B)$ the homotopy category of bounded chain complexes of finitely generated $(A,B)$-bimodules. If $M$ is an $(A,B)$-bimodule then $M^\circ$ is a $(B,A)$-bimodule. If $X_*$ is a chain complex of $(A,B)$-bimodules then $X_*^\circ$ is a chain complex of $(B,A)$-bimodules. For any integer $i$ and any $(A,B)$-bimodule $M$ we write $M[i]$ for the chain complex with term $M$ in degree $i$ and terms $\{0\}$ in all other degrees.
Recall that for a semisimple $k$-algebra $A$ and any bounded chain complex $X_*$ of finitely generated $A$-modules, one has $X_*\cong H(X_*)$ in $K(A)$, where $H(X_*)\in K(A)$ is the $\ZZ$-graded $A$-module consisting of the homology of $X_*$ with trivial boundary maps.

\begin{proposition}\label{prop homotopy consequence}
Let $k$ be a field, let $A$ and $B$ be semisimple $k$-algebras, let $X_*$ be a bounded chain complex of finitely generated $(A,B)$-bimodules satisfying $X_*\otimes_B X_*^\circ\cong A[0]$ in $K(A,A)$ and $X_*^\circ\otimes_A X_*\cong B[0]$ in $K(B,B)$. If $W$ is a simple $B$-module then $X_*\otimes_B W\cong V[i]$ in $K(A)$ for a simple $A$-module $V$ and an integer $i$. Moreover, $\End_A(V)\cong\End_B(W)$ as $k$-algebras.
\end{proposition}

\begin{proof}
Note that $X_*\otimes_B-\colon K(B)\to K(A)$ is an equivalence. Moreover, for any simple $B$-module $W$, we have isomorphisms
\begin{align*}
  \End_B(W) & \cong \Hom_{K(B)}(W[0],W[0]) 
                        \cong \Hom_{K(A)}(X_*\otimes _B W[0],X_*\otimes_B W[0])\\
  & \cong \Hom_{K(A)}(H(X_*\otimes_B W[0]), H(X_*\otimes_B W[0])) 
      \cong \prod_{i\in\ZZ}\End_A(H_i(X_*\otimes_B W[0]))
\end{align*}
of $k$-algebras. Since $\End_B(W)$ is a division algebra over $k$, there exists a unique $i\in\ZZ$ such that $\End_A(H_{i}(X_*\otimes_B W[0]))$ is a division algebra isomorphic to $\End_B(W)$ and $H_j(X_*\otimes_B W[0])=0$ for all $j\in \ZZ$ with $j\neq i$. Thus, $X_*\otimes_B W[0]\cong H(X_*\otimes_B W[0]) \cong H_i(X_*\otimes_B W[0])\cong V[i]$ in $K(A)$ for $V:=H_i(X_*\otimes_B W[0])\in\lmod{A}$. Since $\End_A(V)$ is a division algebra, $V$ is a simple $A$-module and the result follows.
\end{proof}

\begin{proof} {\em of Theorem~\ref{thm app}(c).}\quad Let $X_*$ be a splendid Rickard equivalence between $\ZZ_p G\btilde$ and $\ZZ_p H\ctilde$, set $\gamma:=\sum_{i\in\ZZ} (-1)^i [X_i]\in T^\Delta(\ZZ_p G\btilde, \ZZ_p H\ctilde)$, and let $\gamma_\calO\in T^\Delta(\calO G \btilde,\calO H\ctilde)$ denote the element arising from $\gamma$ via scalar extension from $\ZZ_p$ to $\calO$. Then $\gamma_\calO$ is a $p$-permutation equivalence between $\calO G\btilde$ and $\calO H\ctilde$. 
Set $\mu_K:=\kappa_{G\times H}(\gamma_\calO)\in R(KG\btilde,KH\ctilde)$, a perfect isometry between $\calO G\btilde$ and $\calO H\ctilde$. Let $\alpha\colon \Irr(KH\ctilde)\myiso\Irr(KG\btilde)$ denote the bijection induced by $\mu_K$.

For fixed $r\in\{1,\ldots,p-1\}$ and $h\in\QQ/\ZZ$, the bijection $\alpha$ restricts to a bijection
\begin{equation}\label{eqn alpha}
   \alpha\colon \Irr_0(KH\ctilde,r,h)\myiso\Irr_0(KG\btilde,r,h)
\end{equation}
by Theorem~\cite[Th\'eor\`eme~1.5(2)]{Broue1990} and Proposition~\ref{prop homotopy consequence} applied to $k=\QQ_p$, $A=\QQ_p G\btilde$, $B:=\QQ_p H\ctilde$, and the chain complex $K\otimes_{\ZZ_p} X_*$.

Since $\mu_K$ is the scalar extension from $\QQ_p$ to $K$ of $\kappa_{G\times H}(\gamma)$, the bijection $\alpha$ in (\ref{eqn alpha}) is $\Gamma$-invariant. Thus, for every $\Delta\le\Gamma_b$ one has
\begin{equation*}
   |\Irr_0(KG\btilde,r,h)^\Delta| = |\Irr_0(KH\ctilde,r,h)^\Delta|\,.
\end{equation*}
Note that if $\chi$ and $\chi'$ in $\Irr(KG\btilde)$ are $\Gamma$-conjugate then $h(\chi)=h(\chi')$. With this and the same arguments in the proof of Part~(b), one obtains
\begin{equation*}
      |\Irr_0(KGb,r,h)^\Delta| = |\Irr_0(KHc,r,h)^\Delta|\,.
\end{equation*}
This implies that the $\Gamma_b$-sets $\Irr_0(KGb,r,h)$ and $\Irr_0(KHc,r,h)$ are isomorphic. Taking the disjoint union over the various values of $r$ and $h$, one obtains the second statement of Part~(c). The firest statement follows again as in the proof of Part~(b).
\end{proof}


\end{document}